\documentclass[a4paper,10pt]{amsart}

\usepackage[utf8x]{inputenc}
\usepackage{amsfonts}
\usepackage{amsmath}
\usepackage{amssymb}
\usepackage{graphicx}
\usepackage[all]{xy}
\usepackage{mathtools}
\usepackage{todonotes}

\usepackage[draft=false]{hyperref}  

\DeclareMathOperator{\Ker}{Ker}

\DeclareMathOperator{\id}{id}

\DeclareMathOperator{\End}{End}

\DeclareFontFamily{OT1}{pzc}{}
\DeclareFontShape{OT1}{pzc}{m}{it}{<-> s * pzcmi7t}{}
\DeclareMathAlphabet{\mathpzc}{OT1}{pzc}{m}{it}

\DeclareMathOperator{\FP}{FP}
\DeclareMathOperator{\F}{F}

\newcommand*{\RR}{\ensuremath{\mathbb{R}}}

\newcommand*{\ZZ}{\ensuremath{\mathbb{Z}}}
\newcommand*{\NN}{\ensuremath{\mathbb{N}}}

\newcommand*{\longtwoheadrightarrow}{\ensuremath{\relbar\joinrel\twoheadrightarrow}}

\usepackage{amsthm}
\newtheorem{Lemma}{Lemma}[section]
\newtheorem{Theorem}[Lemma]{Theorem}
\newtheorem{Cor}[Lemma]{Corollary}
\newtheorem{Prop}[Lemma]{Proposition}

\theoremstyle{definition}

\newtheorem{Question}[Lemma]{Question}

\theoremstyle{remark}
\newtheorem{Remark}[Lemma]{Remark}

\usepackage{tikz}
\usepackage{xfrac}
\newcommand{\llangle}{\langle\langle}
\newcommand{\rrangle}{\rangle\rangle}

\DeclareMathOperator{\Sym}{Sym}
\DeclareMathOperator{\Alt}{Alt}
\newcommand{\lex}{{\rm lex}}
\DeclareMathOperator{\Stab}{Stab}
\DeclareMathOperator{\Ends}{Ends}

\DeclareMathOperator{\supp}{supp}
\DeclareMathOperator{\Sim}{Sim}
\usetikzlibrary{arrows}
\tikzset{
  treenode/.style = {align=center, inner sep=0pt},
  arn_n/.style = {treenode, circle, black, draw=black, text width=1.5em},
}

\title[Quasi-automorphisms of $\mathcal{T}_{2,c}$]{Quasi-automorphisms of the infinite rooted $2$-edge-coloured binary tree}

\author{Brita E.A. Nucinkis}
\email{Brita.Nucinkis@rhul.ac.uk}
\address{Department of Mathematics, Royal Holloway, University of London, Egham, TW20 0EX}

\author{Simon St.~John-Green}
\email{Simon.StJG@gmail.com}
\address{Department of Mathematics, University of Southampton, SO17 1BJ, UK}

\date{\today}


\begin{document}

\numberwithin{equation}{section}

\begin{abstract}
We study the group $QV$, the self-maps of the infinite $2$-edge coloured binary tree which preserve the edge and colour relations at cofinitely many locations.  We introduce related groups $QF$, $QT$, $\tilde{Q}T$, and $\tilde{Q}V$, prove that $QF$, $\tilde{Q}T$, and $\tilde{Q}V$ are of type $\F_\infty$, and calculate finite presentations for them.  We calculate the normal subgroup structure of all $5$ groups, the Bieri--Neumann--Strebel--Renz invariants of $QF$, and discuss the relationship of all $5$ groups with other generalisations of Thompson's groups.
\end{abstract}

\maketitle

\section{Introduction}

R.~Thompson's groups $F$, $T$, and $V$ were originally used to, among other things, construct finitely presented groups with unsolvable word problem \cite{Thompson-EmbeddingIntoFinitelyGeneratedSimpleGroups,McKenzieThompson-AnElementaryConstructionOfUnsolvableWordProblems}.  Brown and Geoghegan proved $F$ is $\FP_\infty$ \cite{BrownGeoghegan-AnInfiniteDimensionTFFPinftyGroup}, giving the first example of a torsion-free $\FP_\infty$ group with infinite cohomological dimension, and later Brown used an important new geometric technique to prove that $F$, $T$, and $V$ are all of type $\F_\infty$ \cite{Brown-FinitenessPropertiesOfGroups}.  For more background on Thompson's groups, see \cite{CannonFloydParry-IntroNotesOnThompsons}.

Recently many variations on Thompson's groups have been studied, and using techniques similar to Brown's, many of these have been proven to be of type $\F_\infty$.  These include Brin's groups $sV$ \cite{FluchMarschlerWitzelZaremsky-sVisFinfty}, generalisations of these \cite{MartinezPerezMatucciNucinkis-PresentationsForV}, and the braided Thompson's groups $BV$ and $BF$ \cite{BuxEtAl-BraidedThompsonAreFinfty} which were originally defined by Brin \cite{Brin-AlgebraOfStrandSplittingI}.

Let $\mathcal{T}_{2,c}$ denote the infinite binary $2$-edge-coloured tree and let $QV$ be the group of all bijections on the vertices of $\mathcal{T}_{2,c}$ which respect the edge and colour relations except for possibly at finitely many locations.  The group $QV$ was studied by Lehnert \cite{Lehnert-Thesis} who proved that $QV$ is co-context-free, i.e.~the co-word problem in $QV$ is context-free.  He also gave an embedding of $V$ into $QV$.  More recently, Bleak, Matucci, and Neunh\"offer gave an embedding of $QV$ into $V$ \cite{BleakMatucciNeunhoffer-EmbeddingsIntoVAndCocf}.  The group $QV$ was called $\operatorname{QAut}(\mathcal{T}_{2,c})$ by Bleak, Matucci, and Neunh\"offer, and was called both $\operatorname{QAut}(\mathcal{T}_{2,c})$ and $G$ by Lehnert.

We write $QV$ instead of $\operatorname{QAut}(\mathcal{T}_{2,c})$ in order to follow the convention of naming varations of Thompson's groups using a prefix (e.g.\ $BV$ for the braided Thompson's groups \cite{BuxSonkin-SomeRemarksOnBV} and $nV$ for the Brin--Thompson groups \cite{Brin-HigherDimensionalThompsonGroups}).  The $Q$ may be thought of as standing for ``quasi-automorphism''.  

There is a natural surjection $\pi: QV \longrightarrow V$ (Lemma \ref{lemma:pi is homomorphism}), and we denote by $QF$ the preimage $\pi^{-1}(F)$ and by $QT$ the preimage $\pi^{-1}(T)$.  The group $QF$ was studied by Lehnert, where it is called $G^0$ \cite[Definition 2.9]{Lehnert-Thesis}.  The surjection $\pi:QV \to V$ is not split (Proposition \ref{prop:QT extension doesnt split}), although it does split when restricted to $QF$ (Lemma \ref{lemma:QF SES splits}).   This situation is similar to that of the braided Thompson groups---$BF$ splits as an extension of $F$ by the infinite pure braid group $P_\infty$, whereas $BV$ is a non-split extension of $V$ by $P_\infty$ \cite[\S 1.2]{Zaremsky-NormalSubgroupsOfBraidedThompsonsGroups}.  There is no braided analogue of $T$.

The following theorem appeared in Lehnert's thesis, \cite[Satz 2.14, 2.30]{Lehnert-Thesis}, although his proof that $QF$ is of type $\FP_\infty$ uses Brown's criterion, a different approach from that used here. 

\theoremstyle{plain}\newtheorem*{CustomThmB}{Theorems \ref{theorem:QF presentation} and \ref{theorem:QF is Finfty}}
\begin{CustomThmB}
$QF$ is $\F_\infty$ and has presentation
 \[
  QF = \left\langle \sigma, \alpha, \beta \, \left\vert 
    \begin{array}{l}
     \sigma^2,\, [\sigma, \sigma^{\alpha^2}],\, (\sigma \sigma^\alpha)^3,\, \sigma\sigma^\alpha \sigma \sigma^{\alpha\beta^{-1}\alpha^{-1}}, \\\relax
     [\alpha\beta^{-1},\, \alpha^{-1}\beta\alpha],\, [\alpha\beta^{-1}, \alpha^{-2}\beta\alpha^2], \\\relax
      [\nu, \sigma] \text{ for all $\nu \in X$.}
    \end{array}
  \right.\right\rangle
 \]
 where
 \[
  X = \left\{\begin{array}{c}
\beta, \beta^\alpha, \alpha^3\beta^{-1}\alpha^{-2}, \alpha^{2}\beta^2\alpha^{-1}\beta^{-1}\alpha\beta^{-1}\alpha^{-2} \\
\alpha\beta^2\alpha^{-1}\beta^{-1}\alpha\beta^{-1}\alpha^{-1}, \alpha\beta\alpha^{-1}\beta^2\alpha^{-1}\beta^{-1}\alpha\beta^{-1}\alpha\beta^{-1}\alpha^{-1}                                                          
\end{array} \right\}.
 \]
\end{CustomThmB}

We write $\mathcal{T}_{2,c} \cup \{\zeta\}$ for the disjoint union of $\mathcal{T}_{2,c}$ and a single vertex labelled $\zeta$.  We write $\tilde{Q}V$ for the group of all bijections on the vertices of $\mathcal{T}_{2,c} \cup \{\zeta\}$ which respect the edge and colour relations except for possibly at finitely many locations.  Again, there is a surjection $\pi: \tilde{Q}V \to V$ and we write $\tilde{Q}T$ for the preimage $\pi^{-1}(T)$.  The groups $\tilde{Q}V$ and $\tilde{Q}T$ are easier to study than $QV$ and $QT$ because $\pi : \tilde{Q}V \to V$ splits (Lemma \ref{lemma:QAutX SES is split}). 

In Section \ref{section:finite presentations} we find finite presentations for $\tilde{Q}V$ and $\tilde{Q}T$ and in Section \ref{section:type Finfty} we show $\tilde{Q}V$, $\tilde{Q}T$ and $QF$ are of type $\F_\infty$.  

Section \ref{section:normal subgroups} contains a complete classification of the normal subgroups of $QF$, $QT$, $QV$, $\tilde{Q}T$, and $\tilde{Q}V$.  We find that the structure is very similar to that of $F$, $T$, and $V$.

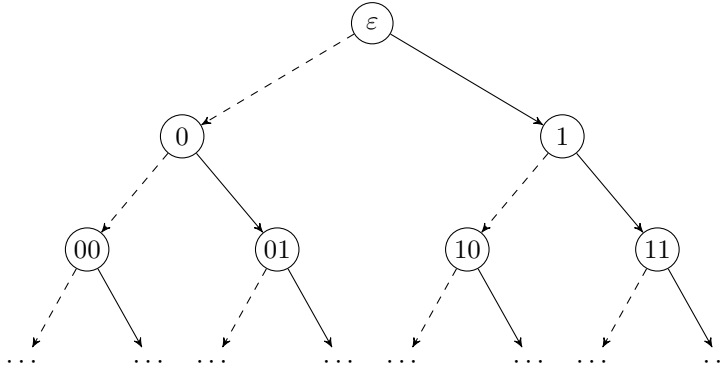
\begin{figure}[ht]
\begin{tikzpicture}[->,>=stealth',level/.style={sibling distance = 5cm/#1,
  level distance = 1.5cm}] 
\node[arn_n] {$\varepsilon$}
    child [dashed]  { node [arn_n,solid] {0} 
      child [dashed]  { node [arn_n,solid] {00} 
        child [dashed]  { node [solid] {$\cdots$} }
        child [solid] { node [solid] {$\cdots$} }
      }
      child [solid]  { node [arn_n,solid] {01}
        child [dashed] { node [solid] {$\cdots$} }
        child [solid]{ node [solid] {$\cdots$} }
      }                         
    }
    child [solid] { node [arn_n,solid] {1}
      child [dashed] { node [arn_n,solid] {10} 
        child [dashed] { node [solid] {$\cdots$} }
        child [solid]{ node [solid] {$\cdots$} }
      }
      child [solid]{ node [arn_n,solid] {11}
        child [dashed] { node [solid] {$\cdots$} }
        child [solid]{ node [solid] {$\cdots$} }
      }                         
    }
; 
\end{tikzpicture}
 \caption{$\mathcal{T}_{2,c}$, the rooted $2$-edge-coloured infinite binary tree.}\label{figure:T2c}
\end{figure}

We identify the vertices of $\mathcal{T}_{2,c}$ with the set $\{0,1\}^*$ of finite length words on the alphabet $\{0,1\}$, denoting the empty word by $\varepsilon$ (see Figure \ref{figure:T2c}), and hence view $QV$ as acting on $\{0,1\}^*$.  Let $\Sym(\{0,1\}^*)$ denote the finite support permutation group on $\{0,1\}^*$ and $\Alt(\{0,1\}^*)$ the finite support alternating group on $\{0,1\}^*$.  Both these groups are normal in $QV$, $QT$, and $QF$.  Let $Z = \{0,1\}^* \cup \{ \zeta \} $, so that $Z$ can be identified with the set of vertices of $\mathcal{T}_{2,c} \cup \{\zeta\}$, then $\Sym(Z)$ and $\Alt(Z)$ have the obvious definition and are normal in $\tilde{Q}V$ and $\tilde{Q}T$.  Recall that $F / [F,F] \cong \ZZ \oplus \ZZ$, but since $T$ and $V$ are simple and non-abelian, $[T,T] \cong T$ and $[V,V] \cong V$ \cite[Theorems 4.1, 5.8, 6.9]{CannonFloydParry-IntroNotesOnThompsons}.

\theoremstyle{plain}\newtheorem*{CustomThmD}{Theorem \ref{theorem:alternative for normal subgroups}}
\begin{CustomThmD}[Normal subgroup structure]
~
\begin{enumerate}
 \item A non-trivial normal subgroup of $QF$ is either $\Alt(\{0,1\}^*)$, $\Sym(\{0,1\}^*)$, or contains 
 \[
 [QF,QF] = \Alt(\{0,1\}^*) \rtimes [F,F].
 \]
 \item A proper non-trivial normal subgroup of $\tilde{Q}T$ is either $\Alt(Z)$, $\Sym(Z)$, or 
 \[
 [\tilde{Q}T, \tilde{Q}T] = \Alt(Z) \rtimes T.
 \]
 \item A proper non-trivial normal subgroup of $\tilde{Q}V$ is either $\Alt(Z)$, $\Sym(Z)$, or 
 \[
 [\tilde{Q}V, \tilde{Q}V] = \Alt(Z) \rtimes V.
 \]
 \item A proper non-trivial normal subgroup of $QT$ is one of either $\Alt(\{0,1\}^*)$, $\Sym(\{0,1\}^*)$, or 
 \[
 [QT,QT] = (\Alt(Z) \rtimes T) \cap QT.
 \]  
 Moreover, $(\Alt(Z) \rtimes T) \cap QT$ is an extension of $T$ by $\Sym(\{0,1\}^*)$.
 \item A proper non-trivial normal subgroup of $QV$ is one of either $\Alt(\{0,1\}^*)$, $\Sym(\{0,1\}^*)$, or 
 \[
 [QV, QV] = (\Alt(Z) \rtimes V) \cap QV.
 \]  
 Moreover, $(\Alt(Z) \rtimes V) \cap QV$ is an extension of $V$ by $\Sym(\{0,1\}^*)$.
\end{enumerate}
\end{CustomThmD}

In Section \ref{section:BNSR invariants} we calculate the Bieri--Neumann--Strebel--Renz invariants, or Sigma invariants, of $QF$ (they are uninteresting for the groups $QT$, $QV$, $\tilde{Q}T$, and $\tilde{Q}V$ since they all have finite abelianisation).  They have previously been computed for $F$ by Bieri, Geoghegan, and Kochloukova \cite{BieriGeogheganKochloukova-SigmaInvariantsOfF}, and for $BF$ by Zaremsky \cite[\S 3]{Zaremsky-NormalSubgroupsOfBraidedThompsonsGroups}.  We also show that $QF$ is an ascending HNN extension 
$$QF \cong QF(1)\ast_{\theta,t},$$
where $QF \cong QF(1).$

We find that $\Sigma^i(QF) \cong \Sigma^i(F)$ for all $i$.  More precisely, we define $\pi^*\chi = \chi \circ \pi$ for any character $\chi$ of $F$ and prove that $\pi^*$ induces an isomorphism between the character spheres $S(F)$ and $S(QF)$.

Let $\chi_0:F \to \RR$ and $\chi_1:F \to \RR$ be the two characters given by $\chi_0(A) = -1$, $\chi_0(B) = 1$, and $\chi_0(A) = \chi_1(B) = 1$, where $A$ and $B$ are the standard generators of $F$.  

\theoremstyle{plain}\newtheorem*{CustomThmC}{Theorem \ref{theorem:BNSR invariants of QF}}
\begin{CustomThmC}[The Bieri--Neumann--Strebel--Renz invariants of $QF$]
The character sphere $S(QF)$ is isomorphic to $S^1$ and for any ring $R$,
\begin{enumerate}
 \item $\Sigma^1(QF, R) = \Sigma^1(QF) = S(QF) \setminus \{ [\pi^*\chi_0], [\pi^*\chi_1]\}$.
 \item $\Sigma^i(QF, R) = \Sigma^i(QF) = S(QF) \setminus \{ [a \pi^*\chi_0 + b \pi^*\chi_1] : a, b \ge 0 \}$ for all $i \ge 2$.
\end{enumerate}
\end{CustomThmC}

In Section \ref{section:relationship of QF to others} we study the relationship of the groups studied in this note to previously studied families of groups which generalise Thompson's groups and are known to be of type $\F_\infty$.  We show that the groups studied here cannot be proved to be of type $\F_\infty$ using  Farley--Hughes' study of Groups of Local Similarities \cite[Theorem 1.1]{FarleyHughes-FinitenessPropertiesOfSomeGroupsOfLocalSimilarities}, or Mart\'{\i}nez-P\'erez--Nucinkis' study of Automorphism Groups of Cantor Algebras \cite{MartinezNucinkis-GeneralizedThompsonGroups}. We also show that one cannot use  Thumann's study of Operad groups \cite{Thumann-FinitenessForOperadGroups} to determine whether $QV$ is of type $\F_\infty.$

\begin{Remark}[Notation]
A tree diagram representing an element of Thompson's group $V$ is a triple $(L, R, f)$ where $L$ and $R$ are rooted subtrees of the infinite binary rooted tree with equal numbers of leaves, and $f$ is a bijection from the leaves of $L$ to the leaves of $R$.  To represent an element of Thompson's group $F$ we require only the first two elements of the tuple---the bijection is understood to preserve the left-to-right ordering of the leaves.  For further background on tree diagrams see \cite[\S 2]{CannonFloydParry-IntroNotesOnThompsons}.
\end{Remark}

\section*{Acknowledgement}
The authors would like to thank the referees for helpful comments.

\section{Introducing \texorpdfstring{$QF$}{QF}, \texorpdfstring{$QT$}{QT}, \texorpdfstring{$QV$}{QV}, \texorpdfstring{$\tilde{Q}T$}{tQT}, and \texorpdfstring{$\tilde{Q}V$}{tQV}}

\subsection{The group \texorpdfstring{$QV$}{QV}}\label{subsection:QV}
As mentioned in the introduction the group $QV$ is the group of all bijections on the vertices of $\mathcal{T}_{2,c},$  the infinite binary $2$-edge-coloured tree, which respect the edge and colour relations except for possibly at finitely many locations. 
Given some $\tau \in QV$ and $a \in \{0,1\}^*$, let $s_a$ be the maximal suffix such that $a = x_a \cdot s_a$ and $\tau a = y_a \cdot s_a$ for some $x_a, y_a \in \{0,1\}^*$ (here the operation $\cdot$ is concatenation of words).  Bleak, Mattuci and Neunh\"offer prove that the set
\[
 \{(x_a, y_a) : a \in \{0,1\}^*\}
\]
is finite \cite[Claim 8]{BleakMatucciNeunhoffer-EmbeddingsIntoVAndCocf}, and hence that the subset
\[
 M_\tau = \{(x, y) : \text{There are infinitely many $a \in \{0,1\}^*$ with $(x_a, y_a) = (x,y) $}\}
\]
is finite also.  Moreover, they show that the sets 
\[
 \mathcal{L}_\tau =  \{x : (x,y) \in M_\tau\}
\]
and
\[
 \mathcal{R}_\tau = \{y : (x,y) \in M_\tau\}
\]
both form finite complete anti-chains for the poset $\{0,1\}^*$, where $x \le x^\prime$ if $x$ is a subword of $x^\prime$.  Write $b_\tau$ for the bijection 
\[
b_\tau : \mathcal{L}_\tau \longrightarrow \mathcal{R}_{\tau}.
\]
Note that restricting $\tau$ to $\mathcal{L}_\tau$ does not necessarily give the bijection $b_\tau$.  This is because if $\sigma$ is any finite permutation of $\{0,1\}^*$ then $b_{\sigma \circ \tau} = b_\tau$.

Let $L$ and $R$ be the binary trees with leaves $\mathcal{L}_\tau$ and $\mathcal{R}_\tau$ respectively, then $(L, R, b_\tau)$ is a tree-pair diagram for an element of $V$ which we denote by $v_\tau$.

In fact, $v_\tau$ can be viewed as the induced action on the space of ends of the simplicial tree $\mathcal{T}_{2, c}$, the set whose elements are enumerated by the right-infinite words $\{0,1\}^{\NN}$.  With its natural topology the space of ends is homeomorphic to a Cantor set with $v_\tau$ acting as a homeomorphism.  Explicitly, the action of $v_\tau$ on $\{0,1\}^*$ is
\begin{align*}
v_\tau : \{0,1\}^{\NN} &\longrightarrow \{ 0,1\}^{\NN} \\
x \cdot s &\longmapsto y \cdot s \text{ for $(x,y)$ in $M_\tau$.}
\end{align*}
This is well-defined since $\mathcal{L}_\tau$ is a complete anti-chain in the poset $\{0,1\}^*$.

\begin{Lemma}\label{lemma:pi is homomorphism}
 The map $\pi: QV \longrightarrow V$ sending $\tau \mapsto v_\tau$ is a group homomorphism.
\end{Lemma}
\begin{proof}
 For any topological space $X$ let $\Ends(X)$ denote the space of ends of $X$.  
 
 If $\tau \in QV$ then $\tau$ may not be a simplicial automorphism of $\mathcal{T}_{2,c}$, but there will always exist a rooted subtree $L \subseteq \mathcal{T}_{2,c}$ such that 
 \[
  \tau \vert_{\mathcal{T}_{2,c} \setminus L} : \mathcal{T}_{2,c} \setminus L \longrightarrow \mathcal{T}_{2,c} \setminus \tau(L)
 \]
 is a simplicial map.  Thus $\tau$ determines a map 
 \[
  \Ends(\tau) : \Ends(\mathcal{T}_{2,c}) \longrightarrow \Ends(\mathcal{T}_{2,c}).
 \]
This is well-defined---if we choose to remove a different rooted subtree $L^\prime$ instead then, since removing a compact set doesn't affect the ends of a space, we see that 
\[
 \mathcal{T}_{2,c} \setminus (L \cup L^\prime) \longrightarrow \mathcal{T}_{2,c} \setminus \tau(L\cup L^\prime)
\]
and hence 
\[
 \mathcal{T}_{2,c} \setminus (L^\prime) \longrightarrow \mathcal{T}_{2,c} \setminus \tau(L^\prime)
\]
induce the same map on $\Ends(\mathcal{T}_{2,c})$. 

If $\sigma$ is some other element of $QV$ then choosing a large enough finite rooted subtree $L$, both maps in the composition
\[
 \mathcal{T}_{2,c} \setminus{L} \stackrel{\tau}{\longmapsto} \mathcal{T}_{2,c} \setminus \tau({L}) \stackrel{\sigma}{\longmapsto} \mathcal{T}_{2,c} \setminus \sigma \circ \tau({L}) 
\]
are simplicial maps.  Now use that $\Ends(-)$ is a functor.
\end{proof}

The kernel of $\pi$ are those permutations of $\{0,1\}^*$ which are the identity on all but finitely many points, the finite support permutation group $\Sym(\{0,1\}^*)$.  Thus there is a group extension
\[
 1 \longrightarrow \Sym(\{0,1\}^*) \longrightarrow QV \longrightarrow V \longrightarrow 1.
\]
This extension does not split, see Proposition \ref{prop:QT extension doesnt split}.

\subsection{The group \texorpdfstring{$\tilde{Q}V$}{tQV}}
Recall that $\mathcal{T}_{2,c} \cup \{\zeta\}$ denotes the 2-edge-coloured tree $\mathcal{T}_{2,c}$ together with an isolated vertex $\zeta$, so the vertex set of $\mathcal{T}_{2,c} \cup \{\zeta\}$ is $Z =\{0, 1\}^* \cup \{ \zeta \}$.  There is a map $\pi$ from $\tilde{Q}V$ to $V$ defined as previously, whose kernel is $\Sym(Z)$.

Let $\le_{\lex}$ be the total order on $\{0,1\}^*$ defined by following rule:  Given a word $x,y \in \{0,1\}$ let $\tilde{x}, \tilde{y} \in \{0,\frac{1}{2} ,1\}^{\NN}$ be the words obtained by concatenating $x$ and $y$ respectively with the infinite word containing only the symbol $\frac{1}{2}$, now we say that $x \le_{\lex} y$ if and only if $\tilde{x}$ is smaller that $\tilde{y}$ in the lexicographical order where $0 \lneq \frac{1}{2} \lneq 1$.  For example, $0 0 \le_{\lex} 0$ but $01 \ge_{\lex} 0$.  
We extend the $\le_{\lex}$ order to $Z$ so that $\zeta$ is strictly larger than all elements of $\{0,1\}^*$.  

\begin{Remark}\label{remark:order preserving bijection b_tau}
If $T$ is any finite binary tree, there is an order preserving bijection 
\[
 b_T : \operatorname{nodes}(T) \cup \{\zeta\} \longrightarrow \operatorname{leaves}(T).
\]
Moreover, adding a caret to $T$ at any leaf doesn't affect the bijection on the remaining leaves. 
\end{Remark}

\begin{Lemma}\label{lemma:QAutX SES is split}
 The short exact sequence
 \[
 1 \longrightarrow \Sym(Z) \longrightarrow \tilde{Q}V \stackrel{\pi}{\longrightarrow} V \longrightarrow 1
\]
is split.
\end{Lemma}
\noindent Using the computer algebra package GAP Bleak, Matucci, and Neunh\"offer provide the same splitting  \cite[Lemma 15]{BleakMatucciNeunhoffer-EmbeddingsIntoVAndCocf}.
\begin{proof}
 Let $(L, R, f)$ be a tree diagram for an element $v$ of $V$, we define the element $\iota(v)$ of $\tilde{Q}V$ by
 \begin{align*}
 \iota(v) : \{0,1\}^* \cup \{\zeta\} &\longrightarrow  \{0,1\}^* \cup \{\zeta\} \\
                x &\longmapsto \left\{\begin{array}{l l} b_R^{-1} \circ f \circ b_L(x) & \text{ if $x \in \operatorname{nodes}(L) \cup \{\zeta\}$,} \\
                                                f(x) & \text{ if $x \in \operatorname{leaves}(L)$}, \end{array}\right.
 \end{align*}
 and extend $\iota(v)$ onto all $\{0,1\}^*$ by having $\iota(v)$ preserve all edge and colour relations below $\operatorname{leaves}(L)$.
 
 If $(L^\prime, R^\prime, f^\prime)$ is an expansion of $(L, R, f)$ then $(L^\prime, R^\prime, f^\prime)$ determines the same element $\iota(v)$ (use Remark \ref{remark:order preserving bijection b_tau}).  Since any two tree diagrams representing the same element of $V$ have a common expansion this proves $\iota(v)$ is independent of choice of tree diagram.

 We claim that $\iota:v \mapsto \iota(v)$ is a group homomorphism.  Let $v$ and $w$ be two elements of $V$ represented by tree diagrams $(L, R, f)$ and $(R, S, g)$ (by performing expansions we may assume any two elements of $V$ are represented by tree pair diagrams in this form).  Thus the composition of the two elements of $V$ has tree pair diagram $(L, S, g \circ f)$.  Now one verifies that for $x \in \operatorname{nodes}(L) \cup \{\zeta\}$,
 \begin{align*}
  \iota(v) \circ \iota(w) (x) &= b_S^{-1} \circ g \circ b_R \circ b_R^{-1} \circ f \circ b_L(x) \\
  &= b_S^{-1} \circ g \circ f \circ b_L(x) \\
  &= \iota(v w) (x).
 \end{align*}
 Similarly, for $x \in \operatorname{leaves}(L)$, $\iota(v) \circ \iota(w)$ and $\iota(vw)$ agree.
\end{proof}

\begin{Remark}\label{remark:explicit description of alpha beta etc}
For an explicit description of this splitting for the generators of $V$, let $A$, $B$, $C$ and $D$ denote the generators of $V$ as denoted by $A$, $B$, $C$ and $\pi_0$ in \cite[Figures 16 and 17]{CannonFloydParry-IntroNotesOnThompsons}.  Then $\iota(A)$, $\iota(B)$, $\iota(C)$, and $\iota(D)$ are shown in Figures \ref{figure:ia}--\ref{figure:id}.
\end{Remark}

Hereafter we define $\alpha = \iota(A)$, $\beta = \iota(B)$, $\gamma = \iota(C)$, and $\delta = \iota(D)$.

\begin{figure}[ht]
\begin{tikzpicture}[->,>=stealth',level/.style={sibling distance = 2cm/#1,
  level distance = 1cm}] 
  \node [arn_n] at (0,0) {$\varepsilon$}
    child{ node {0} }
    child{ node [arn_n] {1}
       child{ node {10} }
       child{ node {11} }
    }; 
  \node [arn_n] at (2, 0) {$\zeta$};
  \draw (3,-1) -- (4,-1);
  \node [arn_n] at (6,0) {1}
    child{ node [arn_n] {$\varepsilon$} 
       child{ node {0} }
       child{ node {10} }
    }
    child{ node {11} };
 
   \node [arn_n] at (8, 0) {$\zeta$};
\end{tikzpicture}
\caption{Element $\alpha = \iota(A)$.}
\label{figure:ia}
\end{figure}
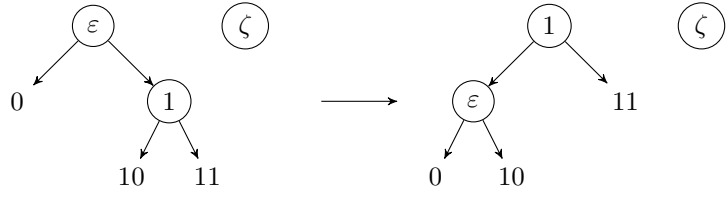

\begin{figure}[ht]
\begin{tikzpicture}[->,>=stealth',level/.style={sibling distance = 2cm/#1,
  level distance = 1cm}] 
  \node [arn_n] at (0,0) {$\varepsilon$}
    child{ node {0} }
    child{ node [arn_n] {1}
       child{ node {10} }
       child{ node [arn_n] {11} 
          child {node {110} }
          child {node {111} }
       }
    }; 
  \node [arn_n] at (2, 0) {$\zeta$};
  \draw (3,-1.5) -- (4,-1.5);
 
  \node [arn_n] at (6,0) {$\varepsilon$}
    child{ node {0} }
    child{ node [arn_n] {11}
       child{ node [arn_n] {1} 
          child {node {10} }
          child {node {110} }
          }
       child{ node {111} }
    }; 
 
   \node [arn_n] at (8, 0) {$\zeta$};
\end{tikzpicture}
\caption{Element $\beta = \iota(B)$.}
\label{figure:ib}
\end{figure}
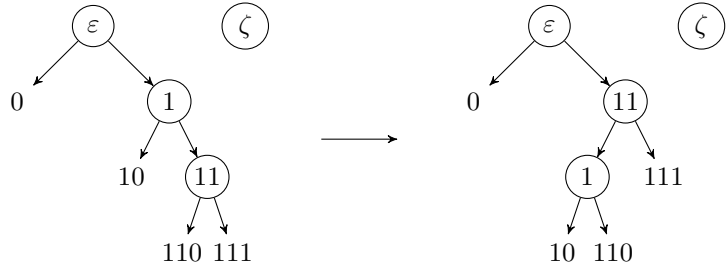

\begin{figure}[ht]
\begin{tikzpicture}[->,>=stealth',level/.style={sibling distance = 2cm/#1,
  level distance = 1cm}] 
  \node [arn_n] at (0,0) {$\varepsilon$}
   child{ node {0} }
    child{ node [arn_n] {1}
       child{ node {10} }
       child{ node {11} }
    }; 
  \node [arn_n] at (2, 0) {$\zeta$};
  \draw (3,-1.5) -- (4,-1.5);
 
  \node [arn_n] at (6,0) {1}
    child{ node {10} }
    child{ node [arn_n] {$\zeta$}
       child{ node {11} }
       child{ node {0} }
    }; 
 
   \node [arn_n] at (8, 0) {$\varepsilon$};
\end{tikzpicture}
\caption{Element $\gamma = \iota(C)$.}
\label{figure:ic}
\end{figure}
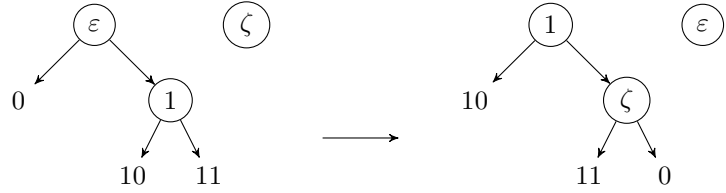

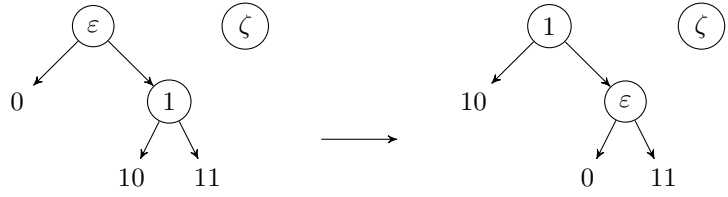
\begin{figure}[ht]
\begin{tikzpicture}[->,>=stealth',level/.style={sibling distance = 2cm/#1,
  level distance = 1cm}] 
  \node [arn_n] at (0,0) {$\varepsilon$}
   child{ node {0} }
    child{ node [arn_n] {1}
       child{ node {10} }
       child{ node {11} }
    }; 
  \node [arn_n] at (2, 0) {$\zeta$};
  \draw (3,-1.5) -- (4,-1.5);
 
  \node [arn_n] at (6,0) {1}
    child{ node {10} }
    child{ node [arn_n] {$\varepsilon$}
       child{ node {0} }
       child{ node {11} }
    }; 
 
   \node [arn_n] at (8, 0) {$\zeta$};
\end{tikzpicture}
\caption{Element $\delta = \iota(D)$.}
\label{figure:id}
\end{figure}

\subsection{The group \texorpdfstring{$QF$}{QF}}\label{subsection:QF}

We define the group $QF$ to be the preimage $\pi^{-1}(F)$ of Thompson's group $F$ under $\pi:QV \rightarrow V$.  

\begin{Lemma}\label{lemma:QF SES splits}
 The short exact sequence 
\[
 1 \longrightarrow \Sym(\{0,1\}^*) \longrightarrow QF \stackrel{\pi}{\longrightarrow} F \longrightarrow 1
\]
is split.
\end{Lemma}

\begin{proof}
 Restrict the splitting map $\iota$ defined in Lemma \ref{lemma:QAutX SES is split} to $QF$.  Since elements $(L, R)$ of $F$ always map the right-hand leaf of $L$ to the right-hand leaf of $R$, we see that $\iota(f)$ fixes $\zeta$ for all $f \in F$.
\end{proof}

\begin{Remark}\label{remark:explicit description of iota for F}
If $(L, R)$ is a tree-diagram for an element $f$ of $F$, where $L$ has nodes $(x_1, \ldots, x_n)$ and $R$ has nodes $(y_1, \ldots, y_n)$ then, assuming $x_i \le_{\lex} x_{i+1}$ and $y_i \le_{\lex} y_{i+1}$ for all $1 \le i \le n-1$, $\iota(f)$ maps $x_i \mapsto y_i$.  The action on all other elements of $\{0,1\}^*$ is determined by $L$ and $R$---if $l_i$ is the i$^\text{th}$ leaf of $L$ and $r_i$ is the $i^\text{th}$ leaf of $R$ then $\iota(f)$ takes the full subtree with root $l_i$ to the full subtree with root $r_i$.  

In particular, for all $f \in F$, the image $\iota(f)$ preserves the $\le_{\lex}$ order on $\{0,1\}^*$.
\end{Remark}

\subsection{The group \texorpdfstring{$QT$}{QT}}
 
We define the group $QT$ to be the preimage $\pi^{-1}(T)$ of Thompson's group $T$ under $\pi:QV \longrightarrow V$.

Restricting the map $\pi$ to $QT$ gives a short exact sequence 
\[
 1 \longrightarrow \Sym(\{0,1\}^*) \longrightarrow QT \stackrel{\pi}{\longrightarrow} T \longrightarrow 1.
\]
 
In \cite[Lemma 2.10]{Lehnert-Thesis}, Lehnert proves that the splitting $\iota:F \to QF$ does not extend to an embedding of $T$ into $QV$.  Examining the proof, he in fact shows the stronger statement below.  Recall that $A$ and $B$ denote the standard generators of $F$.
\begin{Lemma}\label{lemma:standard F embedding doesnt extend}\cite[Lemma 2.10]{Lehnert-Thesis}
 Let $j$ be an embedding of $F$ into $QV$ such that $j(A) = \iota(A)$, then $j$ does not extend to an embedding of $T$ into $QV$.
\end{Lemma}

\begin{Lemma}\label{lemma:splittings agree cofinitely}
 Let $f$ be an element of $F$ and let $j: F \to QF$ be a splitting of $\pi : QF \to F$, then the actions of $\iota(f)$ and $j(f)$ on $\{0,1\}^*$ agree on cofinitely many points.
\end{Lemma}
\begin{proof}
Recall from Section \ref{subsection:QV} that given $\tau \in QV$, for all but finitely many $a \in \{0,1\}^*$ we have $a = x_a \cdot s_a$ for some $x_a \in \mathcal{L}_\tau$ and $\tau(a) = b_{\tau}(x_a) \cdot s_a$.  Since the tree pair diagram for $\pi(\tau)$ is $(L, R, b_{\tau})$, $b_\tau$ is determined completely by $\pi(\tau)$.  We conclude that $\iota(f)$ and $j(f)$ agree on all but finitely many elements of $\{0,1\}^*$, since by assumption $\pi \circ \iota(f) = \pi \circ j(f)$. 
\end{proof}

\begin{Prop}\label{prop:QT extension doesnt split}
 The group extension 
 \[
 1 \longrightarrow \Sym(\{0,1\}^*) \longrightarrow QT \stackrel{\pi}{\longrightarrow} T \longrightarrow 1.
\]
 is not split.
\end{Prop}

\begin{proof}
 Let $j: T \to QT$ be a splitting of $\pi$, by Lemma \ref{lemma:splittings agree cofinitely} we have that $j(A)$ and $\iota(A)$ agree on cofinitely many points. Hence, after possibly extending via adding carets, we can find a tree-pair diagram $(L,R ,f)$ and positive integer $k$ such that in $L$ every leaf under $0$ has length exactly $k$, and every leaf under $1$ has length exactly $k+1$ and similarly in $R$ every leaf under $0$ has length exactly $k+1$, and every leaf under $1$ has length exactly $k$.

 Following the argument of the proof of Lemma \ref{lemma:standard F embedding doesnt extend}, we see that this is incompatible with every possible choice of $j(C)$.
\end{proof}

\section{Finite presentations}\label{section:finite presentations}

In this section we compute finite presentations for the groups $QF$, $\tilde{Q}T$, and $\tilde{Q}V$.  We use a method for finding a finite presentation for a semi-direct product $N \rtimes Q$, whereby we find a presentation for $N$ such that $Q$ acts by permutations and with finitely many orbits on both the generating set and the relating set.  Then as long as $Q$ is finitely presented, and the $Q$-stabilisers of the generating set are finitely generated, one can write down a finite presentation for $N \rtimes Q$ \cite[Lemma A.1]{MartinezPerezMatucciNucinkis-PresentationsForV}.  This technique predates the referenced paper of Mart\'{\i}nez-P\'erez, Matucci, and Nucinkis, but we are not aware of anywhere else where the proof is given. 

We apply this method to $QF$, $\tilde{Q}T$ and $\tilde{Q}V$ using their decompositions as semi-direct products with kernel an infinite symmetric group and quotient $F$, $T$, and $V$ respectively.  

Let $N = \langle S_N \vert R_N \rangle$ and $Q = \langle S_Q \vert R_Q \rangle$ be groups such that $Q$ acts on $S_N$ by permutations and the action induces an action on $R_N$ with $Q(R_N) = R_N$.  Let $S_0$ be a set of representatives for the $Q$-orbits in $S_N$ and $R_0$ a set of representatives for the $Q$-orbits in $R_N$.  
 
For any set $X$, we write $X^*$ for the finite length words of $X$ and we write $q^{-1} s q$ to denote the action of $q \in S_Q$ on $s \in S_0$.
Let $r \in R_0$, then $r$ is a word in $S_N^*$ and hence can be expressed as a word in $(S_0 \cup S_Q)^*$ using the conjugation notation just defined.  Denote the set of these words, one for each $r \in R_0$, by $\widehat{R}_0$.

\begin{Lemma}[{\cite[Lemma A.1]{MartinezPerezMatucciNucinkis-PresentationsForV}}]\label{lemma:MPMNlemmaA.1}
 With the notation above,
 \[
  N \rtimes Q \cong \langle S_0, S_Q \vert \widehat{R}_0, R_Q, [x, s] \text{ for all $s \in S_0$ and all $x \in X_s$} \rangle,
 \]
where $X_s$ is a generating set for $\Stab_Q(s)$ and the semi-direct product is built using the given action of $Q$ on $N$.
\end{Lemma}

\subsection{A finite presentation for \texorpdfstring{$QF$}{QF}}\label{subsection:finite presentation for QF}
 
\begin{Lemma}\label{lemma:transitive action Sigman}
 $F$ acts transitively on 
 \[
  \Sigma_n = \left\{ (x_1, \ldots, x_n) \in \prod_1^n \{0,1\}^* : x_1 \lneq_{\lex} x_2 \lneq_{\lex} \cdots \lneq_{\lex} x_n\right\}
 \]
for all $n \in \NN$, where the action is via the splitting $\iota:F \longrightarrow QF$, the inclusion $QF \longrightarrow QV$, and the usual action of $QV$ on $\{0,1\}^*$.
\end{Lemma}
\begin{proof}
 For $i \ge 1$, let $0^i$ be the following element of $\{0,1\}^*$
 \[
 0^i = \overbrace{0\cdots 0,}^\text{$i$-times,}
 \]  
 and let $0^0 = \varepsilon$.  For $1 \le i \le n$, let $y_i = 0^{n-i}$.
 
 For any $(x_1, \ldots, x_n) \in \Sigma_n$, we build an element of $\gamma$ of $F$ such that
 \[
  \iota(\gamma) : ( y_1, y_2, \ldots, y_n) \mapsto (x_1, \ldots, x_n)
 \]
 Let $R$ be the smallest rooted subtree of $\mathcal{T}_{2,c}$ containing all the $x_i$ as nodes and let $L_0$ be the smallest subtree of $\mathcal{T}_{2,c}$ containing all the $y_i$ as nodes.
 
 For any two nodes $z_1 \lneq_{\lex} z_2$ in a finite rooted subtree $T$ of $\mathcal{T}_{2,c}$ we define
 \[
 d_T(z_1, z_2) = \lvert \{ z \in \operatorname{nodes} (T) : z_1 \lneq_{\lex} z \lneq_{\lex} z_2 \}\rvert
 \] 
 \[
 d_T(z_1, *) = \lvert \{ z \in \operatorname{nodes} (T) : z_1 \lneq_{\lex} z \}\rvert
 \]
 \[
 d_T(*, z_2) = \lvert \{ z \in \operatorname{nodes} (T) : z \lneq_{\lex} z_2 \}\rvert.
 \]
 Thus $d_{L_0}(*, y_0) = d_{L_0}(y_i, y_{i+1}) = d_{L_0}(y_n, *) = 0$ for all $0 \le i \le n-1$.
 
 Expand $L_0$ to $L$ by adding carets until $d_L(y_i, y_{i+1}) = d_R(x_i, x_{i+1})$ for all $i$.  This is possible because adding a caret at a leaf between $y_i$ and $y_{i+1}$ increases $d_L(y_{i}, y_{i+1})$ by $1$.  Now use the description of Remark \ref{remark:explicit description of iota for F} to check that the element $(L, R)$ maps $(y_1, \ldots, y_n)$ onto $(x_1, \ldots, x_n)$.
\end{proof}

\begin{Lemma}\label{lemma:stab of F on sigma}
~
\begin{enumerate}
 \item The $F$-stabiliser of any element in $\Sigma_n$ is isomorphic to the group $\prod_1^{n+1} F$.

\item The $F$-stabiliser of $(0, \varepsilon) \in \Sigma_2$ has generating set
\[
\operatorname{Stab}_{\Sigma_2}((0 , \varepsilon)) = \left\langle \begin{array}{c}
\beta, \beta^\alpha, \alpha^3\beta^{-1}\alpha^{-2}, \alpha^{2}\beta^2\alpha^{-1}\beta^{-1}\alpha\beta^{-1}\alpha^{-2} \\
\alpha\beta^2\alpha^{-1}\beta^{-1}\alpha\beta^{-1}\alpha^{-1}, \alpha\beta\alpha^{-1}\beta^2\alpha^{-1}\beta^{-1}\alpha\beta^{-1}\alpha\beta^{-1}\alpha^{-1}                                                          
\end{array}\right\rangle.
\]
\end{enumerate}
\end{Lemma}
We use the method in \cite[\S 1]{Belk-Thesis} to translate from tree diagrams to words in the generators $\alpha$ and $\beta$.
\begin{proof}
 Recall that $y_i = 0^{n-i}$ for $0 \le i \le n-1$ and $y^n = \varepsilon$.  For any $x \in \{0,1\}^*$, let $x\cdot \{0,1\}^*$ denote all finite words beginning with $x$.  We calculate the stabiliser of $(y_1, \ldots, y_n) \in \Sigma_n$, since $F$ acts transitively on $\Sigma_n$ this determines the stabilisers of all elements of $\Sigma_n$ up to conjugacy.
 
 We claim the stabiliser of $(y_1, \ldots, y_n)$ is exactly those elements of $F$ for which $\iota(F)$ preserves setwise $y_0 \cdot \{0,1\}^*$ and $y_i \cdot 1 \cdot \{0,1\}^*$ for all $1 \le i \le n$.  
 
 If $f \in F$ and $\iota(f)$ maps some $x \in y_0 \cdot \{0,1\}^*$ to an element in $y_i \cdot 1 \cdot \{0,1\}^*$ for some $1 \le i \le n$ then, since $\iota(f)$ preserves the $\lneq_{\lex}$ ordering, it is impossible for $\iota(f)$ to map $y_1$ to $y_1$ as required.  Similarly, again because $\iota(f)$ respects the $\lneq_{\lex}$ ordering, any $\iota(f)$ setwise preserving these sets will necessarily map $y_1$ to $y_1$.  Repeating this argument for elements in $y_j \cdot 1 \cdot \{0,1\}^*$ completes the proof of $(1)$.
 
 Rephrasing the above, the stabiliser of $(0, \varepsilon)$ is exactly those elements of $F$ which preserve setwise the setwise the subtrees under $00$, $01$, and $1$.
 Those elements preserving, for example the subtree under $00$, and acting trivially elsewhere give a copy of $F$ on that interval.  Since we have three such intervals, the stabiliser is $F \times F \times F$.

 Generators of the elements setwise preserving the subtree under $1$ and acting trivially elsewhere are  $\{\beta, \beta^\alpha \}$ where $\alpha = \iota(A)$ and $\beta = \iota(B)$.
 Generators of the elements setwise preserving the subtree under $00$ and acting trivially elsewhere are $\alpha^3\beta^{-1}\alpha^{-2} $ and $\alpha^{2}\beta^2\alpha^{-1}\beta^{-1}\alpha\beta^{-1}\alpha^{-2}$.  
 Generators of the elements setwise preserving the subtree under $01$ and acting trivially elsewhere are $\alpha\beta^2\alpha^{-1}\beta^{-1}\alpha\beta^{-1}\alpha^{-1}$ and $\alpha\beta\alpha^{-1}\beta^2\alpha^{-1}\beta^{-1}\alpha\beta^{-1}\alpha\beta^{-1}\alpha^{-1} $.
\end{proof}

\begin{Lemma}\label{lemma:presentation of Sym01}
The finite support symmetric group on $\Sym(\{0,1\}^*)$ has presentation
 \[
  \Sym(\{0,1\}^*) = \left\langle \sigma_{x,y} \,
  \begin{array}{c}
  \forall x ,y \in \{0,1\}^* \\ x \lneq_{\lex} y 
  \end{array}
   \,\left\vert 
  \begin{array}{ll}
    \sigma_{x,y}^2 & \forall (x,y) \in \Sigma_2 \\\relax
    [\sigma_{x,y},\sigma_{z,w}] & \forall (x, y, z, w) \in \Sigma_4 \\
    (\sigma_{x,y}\sigma_{y,z})^3 & \forall (x, y, z) \in \Sigma_3 \\\relax
    \sigma_{x,y}\sigma_{y,z}\sigma_{x,y}\sigma_{x,z} & \forall (x, y, z) \in \Sigma_3
  \end{array}
 \right.\right\rangle
 \]
 where $\sigma_{x,y}$ denotes the transposition of $x$ and $y$.
\end{Lemma}
\begin{proof}
 Let $x_1, \ldots, x_n$ be some collection of elements of $\{0,1\}^*$ with $x_i \lneq_{\lex} x_{i+1}$ for all $1 \le i \le n-1$.  The finite symmetric group on $\{x_1, \ldots, x_n\}$ has a well-known presentation
 \[
  \Sym(\{x_1, \ldots, x_n\}) = \left\langle \sigma_{x_1, x_2},\ldots, \sigma_{x_{n-1}, x_n}  \left\vert 
     \begin{array}{l}\sigma_{x_i, x_{i+1}}^2, (\sigma_{x_i, x_{i+1}}\sigma_{x_{i+1}, x_{i+2}})^3, \\\relax
         [\sigma_{x_i, x_{i+1}}, \sigma_{x_j, x_{j+1}}] \text{ $\forall \, \lvert i - j \rvert \ge 2$,}
     \end{array}
  \right\rangle\right.
 \]
where $\sigma_{x_i, x_j}$ denotes the transposition of $x_i$ and $x_j$.  We expand this presentation using Tietze transformations.  First add new generators $\sigma_{x_i, x_{i+2}}$ for each suitable $1 \le i \le n-2$, each such generator can be added with the relator 
\[
\sigma_{x_i, x_{i+1}} \sigma_{x_{i+1}, x_{i+2}} \sigma_{x_i, x_{i+1}}\sigma_{x_i, x_{i+2}} 
\]
Next add new generators $\sigma_{x_i, x_{i+3}}$, each can be added with the relator
\[
\sigma_{x_i, x_{i+2}} \sigma_{x_{i+2}, x_{i+3}} \sigma_{x_i, x_{i+2}}\sigma_{x_i, x_{i+3}} 
\]
and so on adding $\sigma_{x_i, x_{i+j}}$ for all $i$ and increasing $j$, until all the required generators are present.  Finally we can add all relators
\[
\sigma_{x_i,x_j}\sigma_{x_j,x_k}\sigma_{x_i,x_j}\sigma_{x_i,x_k} \forall i \lneq j \lneq k
\]
which are not already present.  We've obtained the presentation
 \[
 \Sym(\{x_1, \ldots, x_n\}) = \left\langle \sigma_{x_i, x_j} \text{ $\forall\, i \lneq j$}  \left\vert 
     \begin{array}{l}\sigma_{x_i, x_j}^2 \text{ $\forall\,i \lneq j$}, \\
     (\sigma_{x_i, x_{j}}\sigma_{x_{j}, x_{k}})^3 \text{ $\forall\,i\lneq j \lneq k$}, \\\relax
         [\sigma_{x_i, x_{j}}, \sigma_{x_k, x_{l}}] \text{ $\forall\, i \lneq j \lneq k \lneq l$} \\\relax
         \sigma_{x_i, x_j}\sigma_{x_j, x_k}\sigma_{x_i, x_j}\sigma_{x_i, x_k} \text{ $\forall\,i \lneq j \lneq k$.}
     \end{array}
  \right\rangle\right.
 \]
 
It is now easy to check that this can be extended to the entire finite support symmetric group $\Sym(\{0,1\}^*).$
\end{proof}

\begin{Prop}[{\cite[\S 3]{CannonFloydParry-IntroNotesOnThompsons}}]\label{prop:F presentation}
Thompson's group $F$ has presentation
\[
 F = \langle \alpha, \beta \,\vert\, [\alpha\beta^{-1}, \alpha^{-1}\beta\alpha], [\alpha\beta^{-1}, \alpha^{-2}\beta\alpha^2]\rangle.
\]
\end{Prop}

\begin{Theorem}\label{theorem:QF presentation}
$QF$ has presentation
 \[
  QF = \left\langle \sigma, \alpha, \beta \, \left\vert 
    \begin{array}{l}
     \sigma^2,\, [\sigma, \sigma^{\alpha^2}],\, (\sigma \sigma^\alpha)^3,\, \sigma\sigma^\alpha \sigma \sigma^{\alpha\beta^{-1}\alpha^{-1}}, \\\relax
     [\alpha\beta^{-1},\, \alpha^{-1}\beta\alpha],\, [\alpha\beta^{-1}, \alpha^{-2}\beta\alpha^2], \\\relax
      [\nu, \sigma] \text{ for all $\nu \in X$.}
    \end{array}
  \right.\right\rangle
 \]
 where
 \[
  X = \left\{\begin{array}{c}
\beta, \beta^\alpha, \alpha^3\beta^{-1}\alpha^{-2}, \alpha^{2}\beta^2\alpha^{-1}\beta^{-1}\alpha\beta^{-1}\alpha^{-2} \\
\alpha\beta^2\alpha^{-1}\beta^{-1}\alpha\beta^{-1}\alpha^{-1}, \alpha\beta\alpha^{-1}\beta^2\alpha^{-1}\beta^{-1}\alpha\beta^{-1}\alpha\beta^{-1}\alpha^{-1}                                                          
\end{array} \right\}
 \]
 and $\sigma = \sigma_{0, \varepsilon}$.
\end{Theorem}
\begin{proof}
We apply Lemma \ref{lemma:MPMNlemmaA.1} to the description of $QF$ as the semi-direct product $\Sym(\{0,1\}^*) \rtimes F$ (Lemma \ref{lemma:QF SES splits}), using the presentation of $\Sym(\{0,1\}^*)$ given in Lemma \ref{lemma:presentation of Sym01}, and the presentation of $F$ given in Proposition \ref{prop:F presentation}.  

Let $f \in F$, then necessarily $f(x) \lneq_{\lex} f(y)$ (see Remark \ref{remark:explicit description of iota for F}) and $f^{-1}\sigma_{x,y}f = \sigma_{f(x), f(y)}$.  Hence by Lemma \ref{lemma:transitive action Sigman}, $F$ acts transitively on the generating set of $\Sym(\{0,1\}^*)$ and acts with three $F$-orbits on the relating set.

Let $\sigma = \sigma_{0,\varepsilon}$, this is a representative of the generating set.  We take the representatives of the relating set,
 \[
  \{ \sigma^2,\, [\sigma, \sigma_{1,11}] ,\, (\sigma \sigma_{\varepsilon,1})^3,\, \sigma\sigma_{\varepsilon,1}\sigma\sigma_{0,1} \}.
 \]
Let $\alpha$ and $\beta$ be the generators of $F$ from Proposition \ref{prop:F presentation}.  We can express the transpositions appearing in the relating set as  
 \[
  \sigma_{1,11} = \sigma^{\alpha^2}
 \]
 \[
  \sigma_{\varepsilon, 1} = \sigma^\alpha
 \]
 \[
  \sigma_{0,1} = \sigma^{\alpha\beta^{-1}\alpha^{-1}}
 \]

Since the stabiliser of $\sigma_{x,y}$ for some $x \lneq_{\lex} y$ is equal to the stabiliser of the pair $(x,y)$ in $\Sigma_2$, Lemma \ref{lemma:stab of F on sigma} gives that the stabiliser of the $F$-action on the generators of $\Sym(\{0,1\}^*)$ is a copy of $F \times F \times F$, generated by the set
\[
X = \left\{ \begin{array}{c}
\beta, \beta^\alpha, \alpha^3\beta^{-1}\alpha^{-2}, \alpha^{2}\beta^2\alpha^{-1}\beta^{-1}\alpha\beta^{-1}\alpha^{-2} \\
\alpha\beta^2\alpha^{-1}\beta^{-1}\alpha\beta^{-1}\alpha^{-1}, \alpha\beta\alpha^{-1}\beta^2\alpha^{-1}\beta^{-1}\alpha\beta^{-1}\alpha\beta^{-1}\alpha^{-1}                                                          
\end{array} \right\}.
\]

In summary, with the notation of Lemma \ref{lemma:MPMNlemmaA.1}, $X_\sigma = X$ is as above, $S_0 = \{\sigma\}$, $S_Q = \{\alpha, \beta\}$, 
\[
\widehat{R}_0 = \{\sigma^2,\, [\sigma,\sigma^{\alpha^2}],\, (\sigma\sigma^\alpha)^3, \, \sigma\sigma^\alpha \sigma \sigma^{\alpha\beta^{-1}\alpha^{-1}} \},
\]
and
\[
 R_Q = \{[\alpha\beta^{-1},\, \alpha^{-1}\beta\alpha],\,  [\alpha\beta^{-1}, \alpha^{-2}\beta\alpha^2] \}.
\]
\end{proof}

\begin{Cor}\label{cor:abelianisation QF}
 \[QF/[QF, QF] \cong \ZZ \oplus \ZZ \oplus C_2\]
\end{Cor}
\begin{proof}
Abelianising the relators we find that the only surviving relation is $\sigma^2$.
\end{proof}

\subsection{A finite presentation for \texorpdfstring{$\tilde{Q}T$}{tQT}}

Recall that $\tilde{Q}T$ denotes the preimage of $T$ under $\pi : \tilde{Q}V \longrightarrow V$.  In this section we compute a finite presentation for $\tilde{Q}T$.

Let $[-,-,-]_{\lex}$ be the cyclic order induced by the $\le_{\lex}$ ordering, so $[x,y,z]_{\lex}$ if and only if $x \le_{\lex} y \le_{\lex} z$ or $y \le_{\lex} z \le_{\lex} x$ or $z \le_{\lex} x \le_{\lex} y$.

The next lemma is an analogue of Lemma \ref{lemma:transitive action Sigman}.

\begin{Lemma}\label{lemma:transitive action Lambdan}
 Let $Z = \{0,1\}^* \cup \{\zeta\}$.  $T$ acts transitively on 
 \[
  \Lambda_n = \left\{ (x_1, \ldots, x_n) \in \prod_1^n Z : [x_1,x_2, \ldots, x_n]_{\lex},\, x_i \neq x_j \text{ for all $i \neq j$} \right\}
 \]
for all $n \in \NN$, where the action is via the splitting $\iota:T \longrightarrow \tilde{Q}T$, the inclusion $\tilde{Q}T \longrightarrow \tilde{Q}V$, and the usual action of $\tilde{Q}V$ on $Z$.
\end{Lemma}

\begin{proof}
As in Lemma \ref{lemma:transitive action Sigman}, let $y_i = 0^{n-i}$ for all $1 \le i \le n$.  We build an element $\gamma$ of $T$ such that $\iota(\gamma)$ maps an arbitrary element $(x_1, \ldots, x_n) \in \Lambda_n$ onto $(y_1, \ldots, y_n)$.

Let $L$ be the smallest rooted subtree of $\mathcal{T}_{2,c}$ such that $L$ contains all the $x_i$ as nodes.
Let $f$ be the cyclic permutation of the leaves of $L$ such that the element $t$ of $T$ represented by $(L,L,f)$ satisfies $\iota(t)x_i \le \iota(t)x_{i+1}$ for all $i$.  

If $\iota(t) x_n \neq \zeta$ then, via Lemma \ref{lemma:transitive action Sigman}, there exists $f \in F$ such that $\iota(ft) x_i = y_i$ for all $i$.  

Assume $\iota(t) x_n = \zeta$.  Let $L^\prime$ be formed from $L$ by adding a caret on the left-hand-most leaf, and let $f^\prime$ be the cyclic permutation of the leaves of $L^\prime$ which sends each leaf to it's immediate left-hand neighbour.  Let $t^\prime$ be the element of $T$ represented by $(L^\prime, L^\prime, f^\prime)$, then $\iota(t^\prime t)x_i \lneq_{\lex} \iota(t^\prime t) x_{i+1}$, and $\iota(t^\prime t) x_n \neq \zeta$.  Once again we may use Lemma \ref{lemma:transitive action Sigman}.
\end{proof}

\begin{Lemma}\label{lemma:stab of T on Lambda2}
~
\begin{enumerate}
 \item The $T$-stabiliser of any element in $\Lambda_n$ is isomorphic to $\prod_1^n F$.
 \item The $T$-stabiliser of $(\varepsilon, \zeta) \in \Lambda_2$ has generating set
\[
\operatorname{Stab}((\varepsilon, \zeta)) = \langle \beta, \beta^\alpha, \alpha^2\beta^{-1}\alpha^{-1}, \alpha\beta^2\alpha^{-1}\beta^{-1}\alpha\beta^{-1}\alpha^{-1} \rangle.
\]
\end{enumerate}
\end{Lemma}
\begin{proof}
 The $T$-stabiliser of $(y_1, \ldots, y_{n-1}, \zeta)$ is exactly the $F$-stabiliser of $(y_1, \ldots, y_{n-1})$, which is isomorphic to $\prod_1^n F$ by Lemma \ref{lemma:stab of F on sigma}, this proves part (1). 

Finally, the stabiliser of the subtree with root $1$ is $\langle \beta, \beta^\alpha \rangle$ and that of the subtree with root $0$ is $\langle \alpha^2\beta^{-1}\alpha^{-1}, \alpha\beta^2\alpha^{-1}\beta^{-1}\alpha\beta^{-1}\alpha^{-1}  \rangle$. 
\end{proof}

The proof of the next lemma is identical to that of Lemma \ref{lemma:presentation of Sym01}.
\begin{Lemma}\label{lemma:presentation of SymZ}
The finite support symmetric group on $\Sym(Z)$ has presentation
 \[
  \Sym(Z) = \left\langle \sigma_{x,y} \,
  \begin{array}{c}
  \forall x ,y \in \Lambda_2
  \end{array}
   \,\left\vert 
  \begin{array}{ll}
    \sigma_{x,y}^2 & \forall (x, y) \in \Lambda_2 \\\relax
    [\sigma_{x,y},\sigma_{z,w}] & \forall (x, y, z, w) \in \Lambda_4 \\
    (\sigma_{x,y}\sigma_{y,w})^3 & \forall (x, y, z) \in \Lambda_3 \\
    \sigma_{x,y}\sigma_{y,z}\sigma_{x,y}\sigma_{x,z} & \forall (x, y, z) \in \Lambda_3 \\
  \end{array}
 \right.\right\rangle
 \]
\end{Lemma}

\begin{Prop}[{\cite[\S 5]{CannonFloydParry-IntroNotesOnThompsons}}]\label{prop:T presentation}
Thompson's group $T$ has presentation 
\[
 T = \left\langle \alpha, \beta, \gamma \left\vert \begin{array}{l} [\alpha\beta^{-1},\, \alpha^{-1}\beta\alpha],\,
      [\alpha\beta^{-1}, \alpha^{-2}\beta\alpha^2], \\\relax
      \gamma^{-1}\beta\alpha^{-1}\gamma\beta, \alpha^{-1}\beta^{-1}\alpha\beta^{-1}\gamma^{-1}\alpha\beta\alpha^{-2}\gamma\beta^2, \\\relax
      \alpha^{-1}\gamma^{-1}(\alpha^{-1}\gamma\beta)^2, \gamma^3,  \end{array} \right.\right\rangle
\] 
\end{Prop}

\begin{Theorem}\label{theorem:tQT presentation}
$\tilde{Q}T$ has presentation
 \[
  \tilde{Q}T = \left\langle \sigma, \alpha, \beta, \gamma \, \left\vert 
    \begin{array}{l}
      \sigma^2,\, [\sigma, \sigma^{\alpha^2}],\, (\sigma \sigma^\alpha)^3,\, \sigma\sigma^\alpha \sigma \sigma^{\alpha\beta^{-1}\alpha^{-1}}, \\\relax
      [\alpha\beta^{-1},\, \alpha^{-1}\beta\alpha],\, [\alpha\beta^{-1}, \alpha^{-2}\beta\alpha^2], \\\relax
      \gamma^{-1}\beta\alpha^{-1}\gamma\beta, \alpha^{-1}\beta^{-1}\alpha\beta^{-1}\gamma^{-1}\alpha\beta\alpha^{-2}\gamma\beta^2, \\\relax
      \alpha^{-1}\gamma^{-1}(\alpha^{-1}\gamma\beta)^2, \gamma^3,  \\\relax
      [\nu, \sigma] \text{ for all $\nu \in X$.}
    \end{array}
  \right.\right\rangle
 \]
 where
 \[
  X = \{ \beta, \beta^\alpha, \alpha^2\beta^{-1}\alpha^{-1}, \alpha\beta^2\alpha^{-1}\beta^{-1}\alpha\beta^{-1}\alpha^{-1} \}
 \]
 and $\sigma = \sigma_{0, \varepsilon}$.
\end{Theorem}
\begin{proof}
 The proof is similar to that of Theorem \ref{theorem:QF presentation}, except using Lemma \ref{lemma:presentation of SymZ} and Proposition \ref{prop:T presentation}.
\end{proof}

\begin{Question}
 Is $QT$ finitely presented?
\end{Question}

\begin{Cor}\label{cor:abelianisation wQT}
 \[
 \tilde{Q}T/[\tilde{Q}T,\tilde{Q}T] \cong C_2 
 \]
\end{Cor}
\begin{proof}
 Abelianising the relators kills the generators $\alpha$, $\beta$, and $\gamma$, leaving only the generator $\sigma$ and the relation $\sigma^2$.
\end{proof}

In Corollary \ref{cor:abelianisations from normal subgroup structure} we show that the abelianisation of $QT$ is also isomorphic to the cyclic group of order $2$.

\subsection{A finite presentation for \texorpdfstring{$\tilde{Q}V$}{tQV}}\label{subsection:finite presentation for QAutX}

In this section we compute a finite presentation for $\tilde{Q}V$.  Recall that $Z = \{0,1\}^* \cup \{\zeta\}$.  We also define
 \[
\beta_n = \alpha^{-(n-1)}\beta\alpha^{n-1} \text{ for all }n \ge 1,
 \]
\[
\gamma_n = \alpha^{-(n-1)}\gamma\beta^{n-1} \text{ for all }n \ge 1,
 \]
\[
\delta_1 = \gamma_2^{-1}\delta\gamma_2,
 \]
\[
\delta_n = \alpha^{-(n-1)}\delta_1 \alpha^{n-1} \text{for all }n \ge 2.
 \]
These definitions will allow us to express the presentation of $\tilde{Q}V$ in a simpler form.  The same definitions appear in \cite[p.13,p.16]{CannonFloydParry-IntroNotesOnThompsons}, where they are called $X_n$, $C_n$, and $\pi_n$ respectively.
 
\begin{Lemma}\label{lemma:transitive action Deltan}
 $V$ acts transitively on 
 \[
  \Delta_n = \left\{(x_1, \ldots, x_n) \in \prod_1^n Z : x_i \neq x_j \text{ for all $i \neq j$} \right\},
 \]
for all $n \in \NN$, where the action is via the splitting $\iota:V \longrightarrow \tilde{Q}V$ and the usual action of $\tilde{Q}V$ on $Z$.
\end{Lemma}
\begin{proof}
Let $(x_1, \ldots, x_n) \in \Delta_n$ and let $L$ be the smallest subtree of $\mathcal{T}_{2,c}$ containing all the $x_i$ as nodes.  Choose a bijection $f$ on the leaves of $L$ such that the element $\gamma$ of $V$ represented by $(L, L, f)$ has $\iota(f)x_i \le_{\lex} \iota(f) x_{i+1}$.  Now use Lemma \ref{lemma:transitive action Lambdan}.
\end{proof}

\begin{Lemma}\label{lemma:stab of V on prodZ}
The $V$-stabiliser of $(\varepsilon, \zeta) \in \Delta_2$ has generating set
\[
\operatorname{Stab}((\varepsilon, \zeta)) = \langle \alpha^\prime, \beta^\prime, \gamma^\prime, \delta^\prime, \lambda, \mu \rangle,
\]
where 
\begin{align*}
 \alpha^\prime &=(\alpha^2)  (\beta^{-1}\gamma^{-1}\alpha\delta\alpha^{-1}\gamma\beta)  (\alpha^{-1}\beta^{-1}), \\
 \beta^\prime &= (\alpha\beta_2^2)(\beta_3^{-1}\beta_2^{-1}\alpha^{-1}), \\
 \gamma^\prime &= (\alpha\beta_2)(\delta\delta_2\delta_1\delta)(\beta_2^{-1}\alpha^{-1}), \\
 \delta^\prime &=   (\alpha\beta_2) (\delta_1\delta_0\delta_1)(\beta_2^{-1}\alpha^{-1}), \\
 \lambda &= \alpha^2\beta^{-1}\alpha^{-1}, \\
 \mu &= \beta.
\end{align*}
See Figures \ref{figure:a prime}--\ref{figure:f} for representatives of the elements $\alpha^\prime$, $\beta^\prime$, $\gamma^\prime$, $\delta^\prime$, $\lambda$, and $\mu$ as tree diagrams.

Between them, the elements $\alpha^\prime$, $\beta^\prime$, $\gamma^\prime$, and $\delta^\prime$ generate the subgroup $V^\prime$ of $V$ which fixes the subtrees under $01$ and $11$.
\end{Lemma}

The method of finding expressions in the generators $\alpha, \beta, \gamma, \delta$ is adapted from that in \cite[\S 1]{Belk-Thesis}, and the applet \cite{Kogan-nVtreesApplet} was used for checking these calculations.

\begin{proof}
Let $T$ be a finite rooted subtree of $\mathcal{T}_{2,c}$ with at least two leaves and let $S$ be the rooted subtree with leaves $00$, $01$, $10$, and $11$.  We denote by $T^\prime$ the subtree obtained by taking the subtree of $T$ with root $0$ and glueing $0$ to the node $00$ of $S$, similarly we take the subtree of $T$ with root $1$ and glue it to the leaf $10$ of $S$.

If $f : \operatorname{leaves}(L) \rightarrow \operatorname{leaves}(R)$ is a bijection between the leaves of finite rooted subtrees of $\mathcal{T}_{2,c}$ then we write $f^\prime : \operatorname{leaves}(L^\prime) \rightarrow \operatorname{leaves}(R^\prime)$ for the corresponding bijection (which fixes the leaves $01$ and $11$).  

The homomorphism 
\begin{align*}
\varphi:  V &\longrightarrow V \\
(L, R, f) &\longmapsto (L^\prime, R^\prime, f^\prime)
\end{align*}
is an injection whose image is $V^\prime$.

Let $\alpha^\prime = \varphi(A)$, $\beta^\prime = \varphi(B)$, $\gamma^\prime = \varphi(C)$, and $\delta^\prime = \varphi(D)$ and let $\lambda$ and $\mu$ be the elements shown in Figures \ref{figure:e} and \ref{figure:f}.  By construction, $\alpha^\prime$, $\beta^\prime$, $\gamma^\prime$, and $\delta^\prime$ generate $V^\prime$.

The elements $\lambda$, $\mu$, and the group $V^\prime$ stabilise $\varepsilon$ and $\zeta$, we claim that they also generate $\Stab_V((\varepsilon, \zeta))$.  
Let $\tau \in V$ be an element of the stabiliser represented by $(L, R, f)$, via Lemma \ref{lemma:conj with lambdamu} we may assume that $\tau$ fixes the subtrees under $01$ and $11$, and thus is an element of $V^\prime$.

Finally, we calculate $\alpha^\prime$, $\beta^\prime$, $\gamma^\prime$, and $\delta^\prime$.  Each is a product of three elements of $V$: the first is an element of $F$ which maps the tree $L$ to a right-vine (a tree formed from the trivial tree by adding carets to the right-hand most leaf only) , the second is the necessary permutation of the leaves on the right vine, and the third is the element of $F$ which maps the right-vine to the tree $R$.  The outcome of these calculations is as shown in the statement of the lemma.  Note that the word length of these elements may not be minimal.
\end{proof}

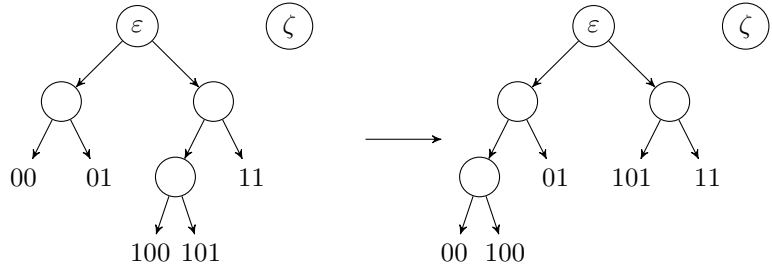
\begin{figure}[ht]
\begin{tikzpicture}[->,>=stealth',level/.style={sibling distance = 2cm/#1,
  level distance = 1cm}] 
  \node [arn_n] at (0,0) {$\varepsilon$}
   child{ node [arn_n] {} 
       child{ node {00} }
       child{ node {01} }
    }
    child{ node [arn_n] {}
       child{ node [arn_n] {} 
          child{ node {100} }
          child{ node {101} }
       }
       child{ node {11} }
    }; 
  \node [arn_n] at (2, 0) {$\zeta$};
  \draw (3,-1.5) -- (4,-1.5);
 
  \node [arn_n] at (6,0) {$\varepsilon$}
    child{ node [arn_n] {} 
       child{ node [arn_n] {} 
         child{ node {00} }
         child{ node {100} }
       }
       child{ node {01} }
    }
    child{ node [arn_n] {}
       child{ node {101} }
       child{ node {11} }
    }; 
 
   \node [arn_n] at (8, 0) {$\zeta$};
\end{tikzpicture}
 \caption{Element $\alpha^\prime = (\alpha^2)  (\beta^{-1}\gamma^{-1}\alpha\delta\alpha^{-1}\gamma\beta)  (\alpha^{-1}\beta^{-1})$ from Lemma \ref{lemma:stab of V on prodZ}.}
\label{figure:a prime}
\end{figure}

\begin{figure}[ht]
\begin{tikzpicture}[->,>=stealth',level/.style={sibling distance = 3.5cm/#1,
  level distance = 1cm}] 
  \node [arn_n] at (0,0) {$\varepsilon$}
   child{ node [arn_n] {} 
       child{ node {00} }
       child{ node {01} }
    }
    child{ node [arn_n] {}
       child{ node [arn_n] {} 
          child{ node {100} }
          child{ node [arn_n] {} 
              child { node {1010} }
              child { node {1011} }
          }
       }
       child{ node {11} }
    }; 
  \node [arn_n] at (2, 0) {$\zeta$};
  \draw (3,-1.5) -- (4,-1.5);
 
  \node [arn_n] at (7,0) {$\varepsilon$}
   child{ node [arn_n] {} 
       child{ node {00} }
       child{ node {01} }
    }
    child{ node [arn_n] {}
       child{ node [arn_n] {} 
          child{ node [arn_n] {} 
              child { node {100} }
              child { node {1010} }
          }
          child{ node {1011} }
       }
       child{ node {11} }
    }; 
    
   \node [arn_n] at (9, 0) {$\zeta$};
\end{tikzpicture}
 \caption{Element $\beta^\prime = (\alpha\beta_2^2)(\beta_3^{-1}\beta_2^{-1}\alpha^{-1})$ from Lemma \ref{lemma:stab of V on prodZ}.}
\label{figure:b prime}
\end{figure}
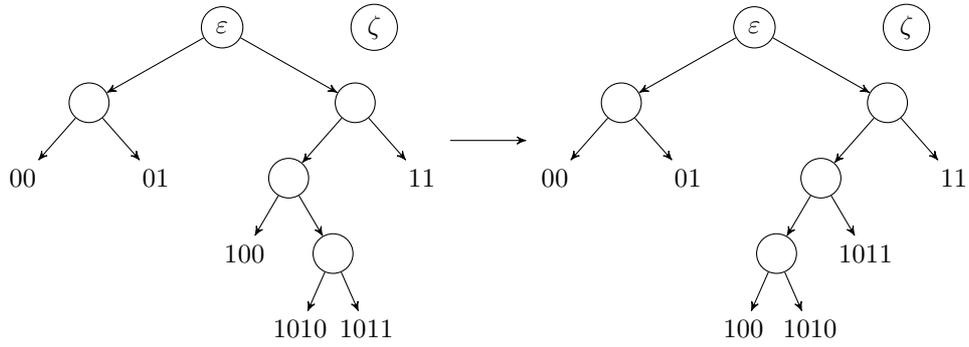

\begin{figure}[ht]
\begin{tikzpicture}[->,>=stealth',level/.style={sibling distance = 2cm/#1,
  level distance = 1cm}] 
  \node [arn_n] at (0,0) {$\varepsilon$}
   child{ node [arn_n] {} 
       child{ node {00} }
       child{ node {01} }
    }
    child{ node [arn_n] {}
       child{ node [arn_n] {} 
          child{ node {100} }
          child{ node {101} }
       }
       child{ node {11} }
    }; 
  \node [arn_n] at (2, 0) {$\zeta$};
  \draw (3,-1.5) -- (4,-1.5);
 
  \node [arn_n] at (6,0) {$\varepsilon$}
   child{ node [arn_n] {} 
       child{ node {100} }
       child{ node {01} }
    }
    child{ node [arn_n] {}
       child{ node [arn_n] {} 
          child{ node {101} }
          child{ node {00} }
       }
       child{ node {11} }
    }; 
    
   \node [arn_n] at (8, 0) {$\zeta$};
\end{tikzpicture}
\caption{Element $\gamma^\prime = (\alpha\beta_2)(\delta\delta_2\delta_1\delta)(\beta_2^{-1}\alpha^{-1})$ from Lemma \ref{lemma:stab of V on prodZ}.}
\label{figure:c prime}
\end{figure}
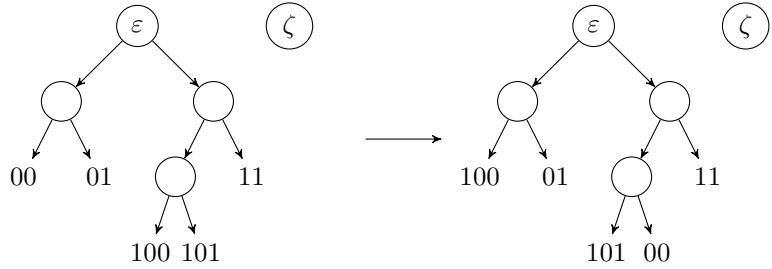

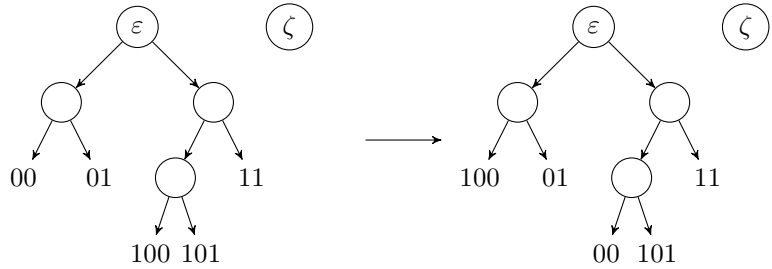
\begin{figure}[ht]
\begin{tikzpicture}[->,>=stealth',level/.style={sibling distance = 2cm/#1,
  level distance = 1cm}] 
  \node [arn_n] at (0,0) {$\varepsilon$}
   child{ node [arn_n] {} 
       child{ node {00} }
       child{ node {01} }
    }
    child{ node [arn_n] {}
       child{ node [arn_n] {} 
          child{ node {100} }
          child{ node {101} }
       }
       child{ node {11} }
    }; 
  \node [arn_n] at (2, 0) {$\zeta$};
  \draw (3,-1.5) -- (4,-1.5);
 
  \node [arn_n] at (6,0) {$\varepsilon$}
   child{ node [arn_n] {} 
       child{ node {100} }
       child{ node {01} }
    }
    child{ node [arn_n] {}
       child{ node [arn_n] {} 
          child{ node {00} }
          child{ node {101} }
       }
       child{ node {11} }
    }; 
    
   \node [arn_n] at (8, 0) {$\zeta$};
\end{tikzpicture}
 \caption{Element $\delta^\prime = (\alpha\beta_2) (\delta_1\delta_0\delta_1)(\beta_2^{-1}\alpha^{-1})$ from Lemma \ref{lemma:stab of V on prodZ}.}
 \label{figure:d prime}
\end{figure}

\begin{figure}[ht]
\begin{tikzpicture}[->,>=stealth',level/.style={sibling distance = 2cm/#1,
  level distance = 1cm}] 
  \node [arn_n] at (0,0) {$\varepsilon$}
   child{ node [arn_n] {} 
       child{ node {10} }
       child{ node [arn_n] {} 
          child{ node {010} }
          child{ node {011} }
       }
    }
    child{ node {1} }
    ; 
  \node [arn_n] at (2, 0) {$\zeta$};
  \draw (3,-1.5) -- (4,-1.5);
 
  \node [arn_n] at (6,0) {$\varepsilon$}
   child{ node [arn_n] {} 
       child{ node [arn_n] {}
          child { node {10} } 
          child { node {010} }
       }
       child{ node {011} }
    }
    child{ node {1} }; 
    
   \node [arn_n] at (8, 0) {$\zeta$};
\end{tikzpicture}
 \caption{Element $\lambda = \alpha^2\beta^{-1}\alpha^{-1}$ from Lemma \ref{lemma:stab of V on prodZ}.}
 \label{figure:e}
\end{figure}
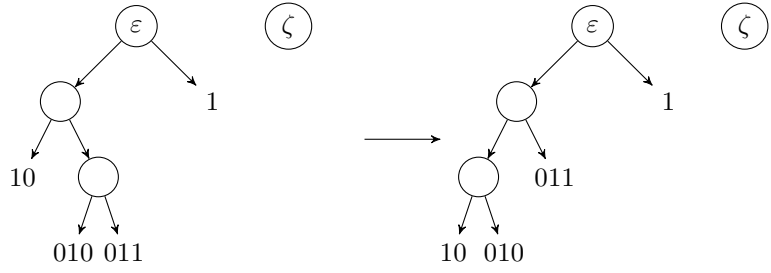

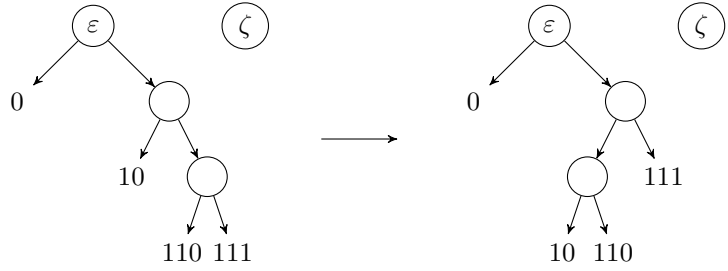
\begin{figure}[ht]
\begin{tikzpicture}[->,>=stealth',level/.style={sibling distance = 2cm/#1,
  level distance = 1cm}] 
  \node [arn_n] at (0,0) {$\varepsilon$}
    child{ node {0} }
    child{ node [arn_n] {}
       child{ node {10} }
       child{ node [arn_n] {} 
          child {node {110} }
          child {node {111} }
       }
    }; 
  \node [arn_n] at (2, 0) {$\zeta$};
  \draw (3,-1.5) -- (4,-1.5);
 
  \node [arn_n] at (6,0) {$\varepsilon$}
    child{ node {0} }
    child{ node [arn_n] {}
       child{ node [arn_n] {} 
          child {node {10} }
          child {node {110} }
          }
       child{ node {111} }
    }; 
 
   \node [arn_n] at (8, 0) {$\zeta$};
\end{tikzpicture}
 \caption{Element $\mu = \beta$ from Lemma \ref{lemma:stab of V on prodZ}.}
 \label{figure:f}
\end{figure}

Let $T$ be a rooted subtree of $\mathcal{T}_{2,c}$ and recall the definition of 
\[
 b_T : \operatorname{nodes}(T) \cup \{\zeta\} \longrightarrow \operatorname{leaves}(T)
\]
from Section \ref{subsection:QV}.  We define $l_\varepsilon(T)$ to be the word length in $\{0,1\}^*$ of the leaf $b_T(\varepsilon)$ and define $l_\zeta(T)$ to be the word length of the leaf $b_T(\zeta)$.  

\begin{Lemma}\label{lemma:conj with lambdamu}
 Let $v \in V$ satisfy $\iota(v)(\zeta) = \zeta$ and $\iota(v)(\varepsilon) = \varepsilon$.  There exist non-negative integers $a$, $b$, $c$, and $d$ such that 
 $\iota( \lambda^{-a}\mu^{-b}v\lambda^c \mu^d )$ fixes the subtrees under $01$ and $11$.
\end{Lemma}
\begin{proof}
Note that $\iota(\lambda)(\varepsilon) = \iota(\mu)(\varepsilon) = \varepsilon$ and $\iota(\lambda)(\zeta) = \iota(\mu)(\zeta) = \zeta$.  To prove the lemma, it suffices to find non-negative integers $a$, $b$, $c$, and $d$ and tree diagram representative $(L, R, f)$ for $\lambda^{-a}\mu^{-b}v\lambda^c \mu^d $ such that 
\[
l_\varepsilon(L) = l_\varepsilon(R) = l_{\zeta}(L) = l_{\zeta}(R) = 2.
\]

Consider the element $\lambda^n$ for some non-negative integer $n$.  This element has tree diagram representative $(L, R, f)$, where $l_{\varepsilon}(L) = n+2$,  $l_{\varepsilon}(R) = 2$, $l_\zeta(L) = 1$, and $l_{\zeta}(R) =1$.
Similarly consider the element $\mu^n$ for some non-negative integer $n$.  This element has tree diagram representative $(L, R, f)$, where $l_{\zeta}(L) = n+2$, $l_{\zeta}(R) = 2$, $l_{\varepsilon}(L) = 1$, and $l_{\varepsilon}(R) = 1$.

Let $v \in V$ be such that $\iota(v)(\zeta) = \zeta$ and $\iota(v)(\varepsilon) = \varepsilon$ and let $(L, R, f)$ be a tree diagram representative for $v$.  If either $l_\zeta(L)$, $l_\zeta(R)$, $l_\varepsilon(L)$, or $l_{\varepsilon}(R)$ are strictly less than $2$ than expand $L$ and $R$ by adding carets until they are equal or greater than $2$.  

Let $a = l_\varepsilon(L) - 2$, $b = l_\zeta(L) - 2$, $c = l_\varepsilon(R) - 2$, and $d = l_\zeta(R) - 2$.  Then, one calculates that $\lambda^{-a}\mu^{-b}v\lambda^c \mu^d$ has tree diagram representative $(L, R, f)$ with 
\[
l_\varepsilon(L) = l_\varepsilon(R) = l_{\zeta}(L) = l_{\zeta}(R) = 2,
\]
as required.
\end{proof}

Modifying Lemma \ref{lemma:presentation of SymZ} slightly we obtain the following.
\begin{Lemma}\label{lemma:presentation of SymZ with duplicated generators}
The finite support symmetric group on $\Sym(Z)$ has presentation
 \[
  \Sym(Z) = \left\langle \sigma_{x,y} \,
  \begin{array}{c}
  \forall x ,y \in \Delta_2 
  \end{array}
   \,\left\vert 
  \begin{array}{ll}
    \sigma_{x,y} = \sigma_{y,x} & \forall (x,y) \in \Delta_2 \\\relax
    \sigma_{x,y}^2  & \forall (x,y) \in \Delta_2  \\\relax
    [\sigma_{x,y},\sigma_{z,w}] & \forall (x, y, z, w) \in \Delta_4 \\
    (\sigma_{x,y}\sigma_{y,w})^3 & \forall (x,y,z) \in \Delta_3 \\
    \sigma_{x,y}\sigma_{y,z}\sigma_{x,y}\sigma_{x,z} & \forall (x, y , z) \in \Delta_3 \\
  \end{array}
 \right.\right\rangle
 \]
\end{Lemma}

\begin{Prop}[{\cite[p.18]{CannonFloydParry-IntroNotesOnThompsons}}]\label{prop:presentation of V}Thompson's group $V$ has presentation
\[
 V = \left\langle  \alpha, \beta, \gamma, \delta : \begin{array}{l} [\alpha\beta^{-1},\, \beta_2],\, [\alpha\beta^{-1}, \beta_3],\,
      \beta \gamma_2\gamma_1^{-1},\, \beta \gamma_3 (\gamma_2 \beta_2)^{-1},\, \gamma_2^2(\gamma_1\alpha)^{-1} \\\relax
      \gamma_1^3,\, \delta_1^2,\, \delta_3\delta_1(\delta_1 \delta_3)^{-1},\, (\delta_2\delta_1)^3,\, \delta_1\beta_3(\beta_3\delta_1)^{-1}, \\\relax
      \beta \gamma_2 \gamma_1 (\gamma_1 \beta_2)^{-1},\, \beta\delta_3 (\delta_2 \beta)^{-1},\, \gamma_3\delta_2(\delta_1\gamma_3)^{-1},\, (\delta_1 \gamma_2)^3
      \end{array} \right\rangle.
\]
\end{Prop}

\begin{Theorem}\label{theorem:QAutX presentation}
$\tilde{Q}V$ has presentation
 \[
  \tilde{Q}V = \left\langle \sigma, \alpha, \beta, \gamma, \delta \, \left\vert 
    \begin{array}{l}
      \sigma \sigma^{\alpha\delta\alpha^{-1}},\,\sigma^2,\, [\sigma, \sigma^{\alpha^2}],\, (\sigma \sigma^\alpha)^3,\, \sigma\sigma^\alpha \sigma \sigma^{\alpha\beta^{-1}\alpha^{-1}}, \\\relax
       [\alpha\beta^{-1},\, \beta_2],\, [\alpha\beta^{-1}, \beta_3],\,
      \beta \gamma_2\gamma_1^{-1},\, \beta \gamma_3 (\gamma_2 \beta_2)^{-1},\, \gamma_2^2(\gamma_1\alpha)^{-1} \\\relax
      \gamma_1^3,\, \delta_1^2,\, \delta_3\delta_1(\delta_1 \delta_3)^{-1},\, (\delta_2\delta_1)^3,\, \delta_1\beta_3(\beta_3\delta_1)^{-1}, \\\relax
      \beta \gamma_2 \gamma_1 (\gamma_1 \beta_2)^{-1},\, \beta\delta_3 (\delta_2 \beta)^{-1},\, \gamma_3\delta_2(\delta_1\gamma_3)^{-1},\, (\delta_1 \gamma_2)^3 \\\relax
      [\nu, \sigma] \text{ for all $\nu \in X$.}
    \end{array}
  \right.\right\rangle
 \]
 where
 \[
  X = \left\{ \begin{array}{l}
 (\alpha^2)  (\beta^{-1}\gamma^{-1}\alpha\delta\alpha^{-1}\gamma\beta)  (\alpha^{-1}\beta^{-1}), (\alpha\beta_2^2)(\beta_3^{-1}\beta_2^{-1}\alpha^{-1}), \\
 (\alpha\beta_2)(\delta\delta_2\delta_1\delta)(\beta_2^{-1}\alpha^{-1}),  (\alpha\beta_2) (\delta_1\delta_0\delta_1)(\beta_2^{-1}\alpha^{-1}), \alpha^2\beta^{-1}\alpha^{-1}, \beta   
  \end{array}
 \right\}
 \]
 and $\sigma = \sigma_{0, \varepsilon}$.
\end{Theorem}
\begin{proof}
 As in Theorems \ref{theorem:QF presentation} and \ref{theorem:tQT presentation}, except using the presentations from Lemma \ref{lemma:presentation of SymZ with duplicated generators} and \ref{prop:presentation of V}.
 
 Explicitly, we have:
 \begin{align*}
  S_Q &= \{\alpha, \beta, \gamma, \delta\} \\
  R_Q &= \left\{ \begin{array}{l} [\alpha\beta^{-1},\, \beta_2],\, [\alpha\beta^{-1}, \beta_3],\,
      \beta \gamma_2\gamma_1^{-1},\, \beta \gamma_3 (\gamma_2 \beta_2)^{-1},\, \gamma_2^2(\gamma_1\alpha)^{-1} \\\relax
      \gamma_1^3,\, \delta_1^2,\, \delta_3\delta_1(\delta_1 \delta_3)^{-1},\, (\delta_2\delta_1)^3,\, \delta_1\beta_3(\beta_3\delta_1)^{-1}, \\\relax
      \beta \gamma_2 \gamma_1 (\gamma_1 \beta_2)^{-1},\, \beta\delta_3 (\delta_2 \beta)^{-1},\, \gamma_3\delta_2(\delta_1\gamma_3)^{-1},\, (\delta_1 \gamma_2)^3
      \end{array} \right\} \\
  S_0 &= \{\sigma\} = \{\sigma_{0,\varepsilon}\} \\
  \hat{R}_0 &= \{\sigma \sigma^{\alpha\delta\alpha^{-1}},\, \sigma^2,\, [\sigma, \sigma^{\alpha^2}],\, (\sigma \sigma^\alpha)^3,\, \sigma\sigma^\alpha \sigma \sigma^{\alpha\beta^{-1}\alpha^{-1}} \} \\
  X_\sigma &= \left\{ \begin{array}{l}
 (\alpha^2)  (\beta^{-1}\gamma^{-1}\alpha\delta\alpha^{-1}\gamma\beta)  (\alpha^{-1}\beta^{-1}), (\alpha\beta_2^2)(\beta_3^{-1}\beta_2^{-1}\alpha^{-1}), \\
 (\alpha\beta_2)(\delta\delta_2\delta_1\delta)(\beta_2^{-1}\alpha^{-1}),  (\alpha\beta_2) (\delta_1\delta_0\delta_1)(\beta_2^{-1}\alpha^{-1}), \alpha^2\beta^{-1}\alpha^{-1}, \beta.         
  \end{array}
 \right\}
 \end{align*}
 
 Compared to the calculation for $\tilde{Q}T$, there is one new element of $R_0$, namely $\sigma \sigma^{\alpha\delta\alpha^{-1}}$ which corresponds to the relation
 \[
 \{\sigma_{x,y} = \sigma_{y,x} \text{ for all $x,y \in \Delta_2$} \}.
 \]
\end{proof}

\begin{Cor}\label{cor:abelianisation wQV}
 \[
\tilde{Q}V/[\tilde{Q}V, \tilde{Q}V] \cong C_2 
 \]
\end{Cor}
\begin{proof}
 Abelianising the presentation leaves the generator $\sigma$ and the relator $\sigma^2$.
\end{proof}

In Corollary \ref{cor:abelianisations from normal subgroup structure} we show that the abelianisation of $QV$ is also isomorphic to $C_2$.

\begin{Question}[{\cite[p.31]{Lehnert-Thesis}}]
 Is $QV$ finitely presented?
\end{Question}

\section{Type \texorpdfstring{$\F_\infty$}{Finfty}}\label{section:type Finfty}

In this section we show that the groups $QF$, $\tilde{Q}T$ and $\tilde{Q}V$ are of type $\F_\infty$.  The idea of the proof for $QF$ is to consider $\Sym(\{0,1\}^*)$ as a countably generated Coxeter group then to form the Davis complex $\mathcal{U}$ (a certain contractible CW-complex on which $\Sym(\{0,1\}^*)$ acts properly), and then extend this to an action of $QF$ on $\mathcal{U}$ which has stabilisers of type $\F_\infty$.  For $\tilde{Q}T$ and $\tilde{Q}V$ we substitute $\Sym(Z)$ for $\Sym(\{0,1\}^*)$.

Let $(W, S)$ be a countably generated Coxeter group, so $W$ is generated by a countable set of involutions $S$.  We start by giving a quick overview of the construction of the Davis complex of a Coxeter group, for background see \cite[\S 5, \S 7]{Davis}.  After this we show that if $Q$ is a group acting by automorphisms on $W$ such that $q(S) = S$ for all $q \in Q$ then there is an action of $W \rtimes Q$ on the Davis complex, where the semi-direct product is formed using the given action of $Q$ on $W$.  This is already well-known, see for example \cite[\S 9.1]{Davis}, so we only give an overview.

For any subset $T$ of $S$ we denote by $W_T$ the subgroup of $W$ generated by $T$.  Recall that a \emph{spherical subset} is a finite subset $T$ of $S$ for which $W_T$ is finite.  The group $W_T$ is known as a \emph{spherical subgroup}.  Let $\mathcal{S}$ denote the poset of spherical subsets in $(W, S)$ and $\mathcal{C}$ the poset of cosets of spherical subgroups, thus elements of $\mathcal{C}$ may be written as $w W_T$ for $w \in W$ and $T \in \mathcal{S}$.  The poset $\mathcal{C}$ admits a left action by $W$,
\[
w^\prime \cdot wW_T = (w^\prime w) W_T.
\]
The \emph{Davis Complex} $\mathcal{U}$ is the geometric realisation of $\mathcal{C}$ and thus admits a left action by $W$ as well.  Note that the $W$-orbits of $n$-simplices in $\mathcal{U}$ are $(n+1)$-element subsets of $S$ generating a finite subgroup of $W$ \cite[p.2]{DicksLeary-SubgroupsOfCoxeterGroups}.

\begin{Prop}\cite[p.3]{DicksLeary-SubgroupsOfCoxeterGroups}
 The Davis complex $\mathcal{U}$ is contractible.
\end{Prop}

Consider a group $Q$ acting by automorphisms on $W$ such that every $q \in Q$ satisfies $q(S) = S$.  We denote by $G$ the semi-direct product $W \rtimes Q$ formed using this action.

The action of $Q$ on $S$ extends to an action of $Q$ on $\mathcal{S}$, by setting 
\[
q \{s_1, \ldots, s_n\} = \{qs_1, \ldots, qs_n\}.
\]
We use this to define a $G$-action on $\mathcal{C}$ by
\[
 (h, q) \cdot wW_T = hw^{q^{-1}} W_{qT}.
\]
One checks that this is well-defined and preserves the poset structure.  The $G$-action on $\mathcal{C}$ induces a $G$-action on $\mathcal{U}$.

The next lemma appears in \cite[Propostion 9.1.9]{Davis}, as does the first part of Proposition \ref{prop:G is Finfty}.

\begin{Lemma}\label{lemma:isotropy of semi-direct construction}
 The $W \rtimes Q$-isotropy subgroup of $wW_T \in \mathcal{C}$ is $(W_{T})^{w^{-1}} \rtimes Q_T$, where $Q_T$ is the $Q$-isotropy of $T \in \mathcal{S}$.
\end{Lemma}

Let $\mathcal{S}_n$ denote the set of unordered $n$-element subsets of $S$, equivalently the spherical subsets of size $n$.  The set $\mathcal{S}_n$ admits a $Q$-action, the restriction of that on $\mathcal{S}$.

\begin{Prop}\label{prop:G is Finfty}
 If, for all positive integers $n$, $Q$ acts on $\mathcal{S}_n$ with finitely many $Q$-orbits and with stabilisers of type $\F_\infty$, then $G$ is of type $\F_\infty$.
\end{Prop}
\begin{proof}
Since the set of $W$-orbits of $n$-simplices in $\mathcal{U}$ is exactly the set $\mathcal{S}_{n+1}$, there are as many $G$-orbits of $n$-simplicies in $\mathcal{U}$ as $Q$-orbits in $\mathcal{S}_{n+1}$.  Thus $G$ acts cocompactly on $\mathcal{U}$.

By construction, the $Q$-isotropy of any point in $\mathcal{U}$ is the $Q$-isotropy of some spherical subgroup $W_T$.  Thus assumption (2) implies that the $Q$-stabiliser of any point has type $\F_\infty$, so combining with Lemma \ref{lemma:isotropy of semi-direct construction} and \cite[Proposition 2.7]{Bieri-HomDimOfDiscreteGroups}, the $G$-stabiliser of any point in $\mathcal{U}$ is $\F_\infty$.  Theorem \ref{thm:contractible finite type with Finfty stab gives Finfty} below completes the proof.
\end{proof}

\begin{Theorem}\label{thm:contractible finite type with Finfty stab gives Finfty}\cite[Theorem 7.3.1]{Geoghegan}
 If there exists a contractible $G$-CW complex which is finite type mod $G$ and has stabilisers of type $\F_\infty$, then $G$ is of type $\F_\infty$.
\end{Theorem}

At this point we specialise to the Coxeter group $\Sym(Z)$.  Let $S$ be the set of unordered $2$-element subsets of $Z$.  This is the set of generators of $\Sym(Z)$ given in Lemma \ref{lemma:presentation of SymZ}.  Recall from Section \ref{subsection:finite presentation for QAutX} that $\Delta_n$ is the set of ordered $n$-element subsets of $Z$, so there is a surjective $V$-map $\Delta_2 \to S$, given by projection.
Recall that $\mathcal{S}_n$ denotes the set of unordered $n$-element subsets of elements of $S$ (note this is not equivalent to the set of unordered $2n$-element subsets of $S$).  There is also a surjective $V$-map $\Delta_{2n} \to \mathcal{S}_n$.

\begin{Lemma}\label{lemma:V acts on Sn with Finfty stab}
  The $V$-stabiliser of $\{(x_1,y_1), \ldots, (x_n, y_n)\} \in \mathcal{S}_n$ is of type $\F_\infty$ for all $\{(x_1,y_1), \ldots, (x_n, y_n)\} \in \mathcal{S}_n$.  
\end{Lemma}
\begin{proof}
 Since $V$ acts transitively on $\mathcal{S}_n$ it is sufficient to check that the stabiliser $\Stab_V(\{(0^{2n-2}, 0^{2n-1}), \ldots, (\varepsilon, \zeta)\})$ is $\F_\infty$.
 
 Let $\{(0^{2n-2}, 0^{2n-1}), \ldots, (\varepsilon, \zeta)\} \in \mathcal{S}_n$, and let $A$ be the preimage of this element under the $V$-map $\Delta_{2n} \to \mathcal{S}_n$.  There is an homomorphism $\pi : \Stab_V(A) \to \Sym_{2n}$ which records the permutation on the elements $\{0^{2n-2}, \ldots, 0, \varepsilon, \zeta\}$.  The kernel of $\pi$ is exactly $\Stab_V(\{(0^{2n-2}, 0^{2n-1}), \ldots, (\varepsilon, \zeta)\})$ which is $\F_\infty$ by Corollary \ref{cor:V acts on Deltan with Finfty stab} so combining with \cite[Proposition 2.7]{Bieri-HomDimOfDiscreteGroups} and the fact that $\Sym_n$ is finite and hence $\F_\infty$, we deduce that $\Stab_V(A)$ is $\F_\infty$.
 
 Observing that $\Stab_V(A) = \Stab_V(\{(0^{2n-2}, 0^{2n-1}), \ldots, (\varepsilon, \zeta)\})$ completes the proof for the $V$-stabiliser. 
\end{proof}

\begin{Theorem}\label{theorem:QAut(X) is Finfty}
 The group $\tilde{Q}V$ is of type $\F_\infty$.
\end{Theorem}
\begin{proof}
We prove the statement for $\tilde{Q}V$ first, using Proposition \ref{prop:G is Finfty}.  Since $V$ acts transitively on $\Delta_{2n}$, it acts transitively on $\mathcal{S}_n$ also.  The stabilisers are of type $\F_\infty$ is Lemma \ref{lemma:V acts on Sn with Finfty stab}.
\end{proof}

Next, we will prove the group $\tilde{Q}T$ is $\F_\infty$, for which we require the following technical lemma.

\begin{Lemma}\label{lemma:stab of T on Deltan}
 $T$ acts with finitely many orbits and stabilisers of type $\F_\infty$ on $\Delta_n$.
\end{Lemma}
\begin{proof}
Let $\operatorname{Cyc}_n$ denote the subgroup of $\Sym_n$ generated by the cyclic permutation $(1,2 \ldots, n)$.  Let $R$ be a set of representatives of the cosets $\Sym_n /\operatorname{Cyc}_n $.  For any $\sigma \in R$, we define a $T$-map 
 \begin{align*}
i_\sigma:\Lambda_n &\longrightarrow \Delta_n  \\
(x_1, \ldots, x_n) &\longmapsto (x_{\sigma(1)}, \ldots, x_{\sigma(n)}).
 \end{align*}
 Taking a product of these gives an $T$-map,
 \[
  \coprod_{\sigma \in R} \Lambda_n \stackrel{\coprod i_\sigma}{\longrightarrow} \Delta_n,
 \]
 which one checks is a bijection.  Since the $T$-stabiliser of any element of $\Lambda_n$ is $\F_\infty$ (Lemma \ref{lemma:stab of T on Lambda2}), the $T$-stabiliser of any element of $\Delta_n$ is also.
\end{proof}

\begin{Theorem}\label{theorem:tQT is Finfty}
 The group $\tilde{Q}T$ is of type $\F_\infty$.
\end{Theorem}
\begin{proof}
Again, we let $S$ be the unordered $2$-element subsets of $Z$.  Using the argument of Lemma \ref{lemma:V acts on Sn with Finfty stab} we deduce from Lemma \ref{lemma:stab of T on Deltan} that $T$ acts with finitely many orbits and with stabilisers of type $\F_\infty$ on $\mathcal{S}_n$.  The statement follows from Proposition \ref{prop:G is Finfty}.  
\end{proof}

To prove that $QF$ is of type $\F_\infty$, we will again use Proposition \ref{prop:G is Finfty}, but now we are considering the $F$-action on the symmetric group $\Sym(\{0,1\}^*)$ so we let $S$ be the unordered $2$-element subsets of $\{0,1\}^*$.  Thus the correct analogue of Lemma \ref{lemma:stab of T on Deltan} is the following.

\begin{Lemma}\label{lemma:stab of F on notquiteDeltan}
 $F$ acts with finitely many orbits and stabilisers of type $\F_\infty$ on 
 \[
  X_n = \{ (x_1, \ldots, x_n) \in \prod_1^n \{0,1\}^* : x_i \neq x_j \,\forall i \neq j\}.
 \]
\end{Lemma}
\begin{proof}
 The proof is similar to that of Lemma \ref{lemma:stab of T on Deltan}, we construct a bijection of $F$-sets
 \[
  \coprod_{\sigma \in \Sym_n} \Sigma_n \stackrel{\coprod i_\sigma}{\longrightarrow} X_n,
 \]
 where $i_\sigma$ is the map
 \begin{align*}
   i_\sigma : \Sigma_n &\longrightarrow X_n \\
   (x_1, \ldots, x_n) &\longmapsto (x_{\sigma(1)}, \ldots, x_{\sigma(n)}).
 \end{align*}
 Now, since the $F$-stabiliser of any element of $\Sigma_n$ is $\F_\infty$ (Lemma \ref{lemma:stab of F on sigma}), the $F$-stabiliser of any element of $X_n$ is also.
\end{proof}

\begin{Theorem}\label{theorem:QF is Finfty}
 The group $QF$ is of type $\F_\infty$.
\end{Theorem}
\begin{proof}
Using the argument of Lemma \ref{lemma:V acts on Sn with Finfty stab} we deduce from Lemma \ref{lemma:stab of F on notquiteDeltan} that $F$ acts with finitely many orbits and stabilisers of type $\F_\infty$ on $\mathcal{S}_n$.  The statement now follows from Proposition \ref{prop:G is Finfty}.
\end{proof}

\subsection{The group \texorpdfstring{$\Stab_V((0^{n-2}, \ldots, 0, \varepsilon, \zeta))$}{Stab} is of type \texorpdfstring{$\F_\infty$}{Finfty}}\label{subsection:Stab is Finfty}
In this section we study the action of $V$ on $\Delta_n$ and prove in Corollary \ref{cor:V acts on Deltan with Finfty stab} that it acts with stabilisers of type $\F_\infty$.

Let $T$ be a rooted subtree of $\mathcal{T}_{2,c}$ and recall the definition of 
\[
 b_T : \operatorname{nodes}(T) \cup \{\zeta\} \longrightarrow \operatorname{leaves}(T)
\]
from Section \ref{subsection:QV}.  For $x \in \{0,1\}^*$ we define $l_x(T)$ to be the word length in $\{0,1\}^*$ of the leaf $b_T(x)$.  This definition has already been seen in Section \ref{subsection:finite presentation for QAutX}, just before Lemma \ref{lemma:conj with lambdamu}.

Let $v \in V$ satisfy $\iota(v)(x) = x$ and let $(L, R, f)$ be a representative for $v$.  We define 
\[
 \tilde\chi_x (v) = l_x(L) - l_x(R).
\]

\begin{Lemma}\label{lemma:xi characters}
 For any $x \in Z$, the map $\tilde\chi_x$ is a group homomorphism $\Stab_V(x) \to \ZZ$.
\end{Lemma}
\begin{proof}
 Since $\tilde\chi_x$ is defined only on $\Stab_V(x)$, the values of $\tilde\chi_x$ are invariant under adding carets.  Thus, since any two tree diagram representatives of an element of $V$ have a common expansion, values taken by $\tilde\chi_x$ don't depend on the tree diagram representative.  Let $v,w $ be any two elements of $\Stab_V(x)$, let $(L, R, f)$ be a tree diagram representative for $v$ and let $(R, S, g)$ be a tree diagram representative for $w$.  Then,
 \[
  \tilde\chi_x(vw) = l_x(L) - l_x(S) = l_x(L) - l_x(R) + l_x(R) - l_x(S) = \tilde\chi_x(f) + \tilde\chi_x(g).
 \]

\end{proof}

Let $T_1$ be the tree with $2$ leaves and let $T_n$ be obtained from $T_{n-1}$ by adjoining carets to the first and second leaves (measuring from smallest to largest using $\le_{\lex}$).  In particular, $T_n$ has $2n$ leaves.  See Figure \ref{figure:Ln} for a picture of $T_3$.  Let $L_n$ be the subgroup of $V$ consisting of elements which fix the full subtrees with root the $2i^\text{th}$-leaf of $T_n$ for all $1 \le i \le n$.  For example, the group $L_3$ is the subgroup of $V$ fixing the full subtrees under $001$, $011$, and $11$.  One checks that $L_n \in \Stab_V((0^{n-2}, \ldots, 0, \varepsilon, \zeta))$.

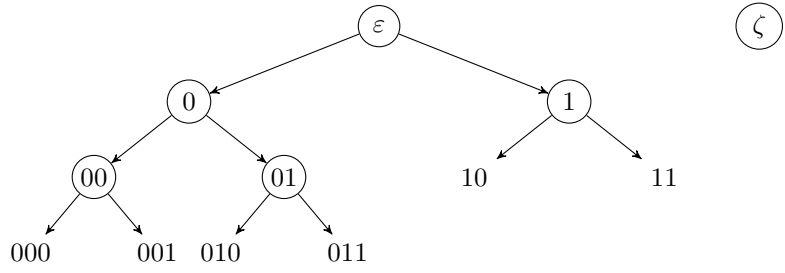
\begin{figure}[ht]
 \begin{tikzpicture}[->,>=stealth',level/.style={sibling distance = 5cm/#1,
  level distance = 1cm}] 
  \node [arn_n] at (0,0) {$\varepsilon$}
    child{ node [arn_n] {0} 
       child{ node [arn_n] {00} 
          child{ node {000} }
          child{ node {001} }
       }
       child{ node [arn_n] {01}
          child{ node {010} }
          child{ node {011} }
       }
    }
    child{ node [arn_n] {1}
       child{ node {10} }
       child{ node {11} }
    }; 
  \node [arn_n] at (5, 0) {$\zeta$};
\end{tikzpicture}
 \caption{The tree $T_3$ from the proof of Lemma \ref{Lemma:Ln iso V}.}\label{figure:Ln}
\end{figure}

\begin{Lemma}\label{Lemma:Ln iso V}
 The group $L_n$ is isomorphic to $V$ for all $n \ge 2$.
\end{Lemma}

\begin{proof}
The group $L_n$ is a copy of $V_{2,n}$, Thompson's group $V$ but on the disjoint union of $n$-trees.  These trees are exactly the full subtrees with roots the $(2i-1)^\text{th}$-leaves of $T_n$ for each $1 \le i \le n$.  There is a result of Higman that $V_{2,n}$ is isomorphic to $V$ \cite{Higman-FinitelyPresentedInfiniteSimpleGroups}.
\end{proof}

For $1 \le i \le n$, let $\lambda_{n,i}$ be the element of $F$ with tree diagram representative $(L_{n,i}, R_{n,i})$ where $L_{n,i}$ is obtained from $T_n$ by adjoining a caret to the $2i^\text{th}$-leaf, and $R_{n,i}$ is obtained from $T_n$ by adjoining a caret to the $(2i-1)^\text{th}$-leaf of $T_n$.  For example, $\lambda_{2,1}$ and $\lambda_{2,2}$ are the elements $\lambda$ and $\mu$ of Section \ref{subsection:finite presentation for QAutX} and $\lambda_{3,2}$ is shown in Figure \ref{figure:lambda32}.

One calculates that:
\begin{Lemma}\label{lemma:properties of tildechi}
For all positive integers $n$, all integers $i$ such that $1 \le i \le n-2$, and all integers $j$ with $1 \le j \le n$,
 \begin{align*}
\tilde\chi_\zeta : L_n &\longmapsto 0, \\\relax
\tilde\chi_\varepsilon : L_n &\longmapsto 0, \\\relax
\tilde\chi_{0^i} : L_n &\longmapsto 0, \\\relax
\tilde\chi_\zeta : \lambda_{n,j} &\longmapsto \left\{ \begin{array}{l l}
                                             1 & \text{if $j = n$,} \\
                                             0 & \text{else,}
                                            \end{array} \right. \\\relax
\tilde\chi_\varepsilon : \lambda_{n,j} &\longmapsto \left\{ \begin{array}{l l}
                                             1 & \text{if $j = n-1$,} \\
                                             0 & \text{else,}
                                            \end{array} \right.\\\relax
\tilde\chi_{0^i} : \lambda_{n,j} &\longmapsto \left\{ \begin{array}{l l}
                                             1 & \text{if $j = n-1-i$,} \\
                                             0 & \text{else.}
                                            \end{array} \right.
\end{align*}
\end{Lemma}

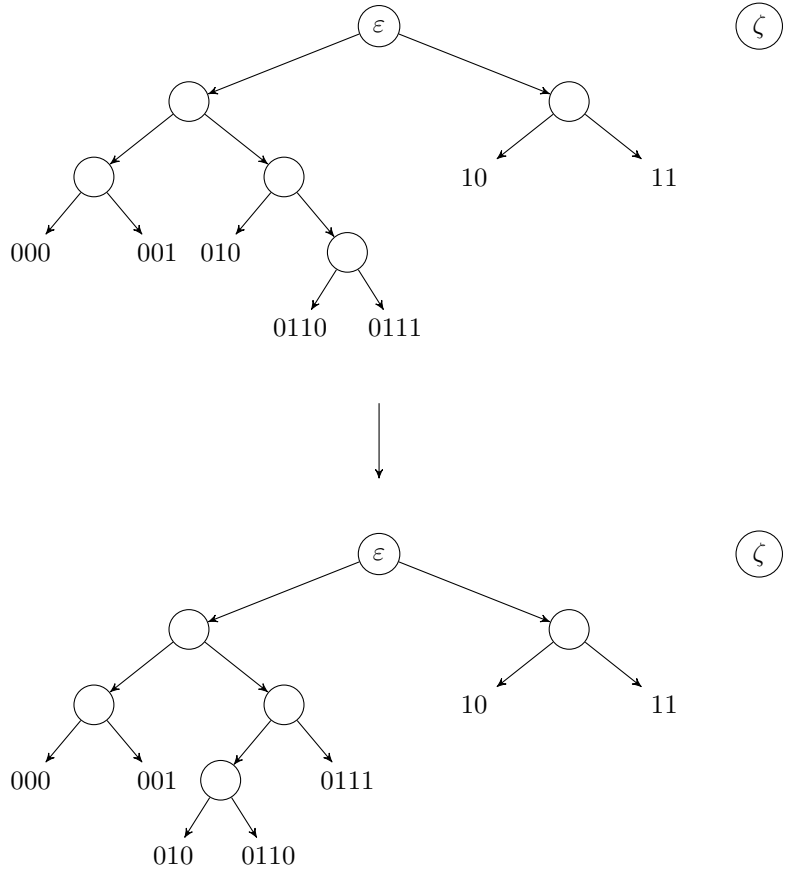
\begin{figure}[ht]
\begin{tikzpicture}[->,>=stealth',level/.style={sibling distance = 5cm/#1,
  level distance = 1cm}] 
  \node [arn_n] at (0,0) {$\varepsilon$}
    child{ node [arn_n] {} 
       child{ node [arn_n] {} 
          child{ node {000} }
          child{ node {001} }
       }
       child{ node [arn_n] {}
          child{ node {010} }
          child{ node [arn_n] {} 
              child{ node {0110} }
              child{ node {0111} }
          }
       }
    }
    child{ node [arn_n] {}
       child{ node {10} }
       child{ node {11} }
    }; 
  \node [arn_n] at (5, 0) {$\zeta$};
  
  \draw (0,-5) -- (0,-6);
  
  \node [arn_n] at (0,-7) {$\varepsilon$}
    child{ node [arn_n] {} 
       child{ node [arn_n] {} 
          child{ node {000} }
          child{ node {001} }
       }
       child{ node [arn_n] {}
          child{ node [arn_n] {}
              child{ node {010} }
              child{ node {0110} }
          }
          child{ node {0111} }
       }
    }
    child{ node [arn_n] {}
       child{ node {10} }
       child{ node {11} }
    }; 
  \node [arn_n] at (5, -7) {$\zeta$};
\end{tikzpicture}
 \caption{$\lambda_{3,2}$}\label{figure:lambda32}
\end{figure}

\begin{Lemma}\label{lemma:conj with lambda n i}
 Let $v \in \Stab_V((0^{n-2}, \ldots, \varepsilon, \zeta))$.  For $1 \le i \le n$, there exist non-negative integers $a_i$ and $b_i$ such that the element
 $\lambda_{n,1}^{-a_1}\cdots \lambda_{n,n}^{-a_n}v\lambda_{n,n}^{b_1} \lambda_{n,1}^{b_n}$ is contained in $L_n$.
\end{Lemma}
\begin{proof}
 In the case $n = 2$ this is Lemma \ref{lemma:conj with lambdamu}, for $n > 2$ the proof is similar.
\end{proof}

\begin{Lemma}\label{lemma:Ln and lambdas generate stab}
 The group $\Stab_V((0^{n-2}, 0^{n-1}, \ldots, \varepsilon, \zeta))$ is generated by $L_n$ and $\lambda_{n,i}$ for $i = 1, \ldots, n$.
\end{Lemma}
\begin{proof}
 This follows from Lemma \ref{lemma:conj with lambda n i}.
\end{proof}

If $G$ is a group, $\theta : G \to G$ an endomorphism, and $\tilde\lambda$ any symbol, then we denote by $G \ast_{\tilde\lambda, \theta} $ the corresponding HNN-extension.  Recall that an HNN-extension is said to be \emph{ascending} if $\theta$ is injective.

\begin{Prop}
\[
\Stab_V((0^{n-2}, 0^{n-2}, \ldots, \varepsilon, \zeta)) \cong L_n \ast_{\tilde\lambda_1, \theta_1} \ast_{\tilde\lambda_2, \theta_2} \cdots \ast_{\tilde\lambda_n, \theta_n},
\]
where the injective homomorphisms $\theta_i$ correspond to conjugation by $\lambda_{n,i}^{-1}$.
\end{Prop}
\begin{proof}
Let $\theta_i : L_n \to L_n$ be the injective homomorphism $v \mapsto \lambda_{n,i}v\lambda_{n,i}^{-1}$, and let
 \[
  \Phi : (V^\prime \ast_{\tilde\mu, \theta_1}) \ast_{\tilde\lambda_i, \theta_2} \longrightarrow \Stab_V((0^{n-2}, 0^{n-1}, \ldots, \varepsilon, \zeta)),
 \]
be the map specialising $\tilde\lambda_i$ to $\lambda_{n,i}^{-1}$.  Since $L_n$ together with the $\lambda_{n,i}$ generate $\Stab_V((0^{n-2}, \ldots, \varepsilon, \zeta))$ (Lemma \ref{lemma:Ln and lambdas generate stab}), the map $\Phi$ is a surjection. Let $v \in \Ker \Phi$, since the HNN-extensions are ascending, we can write 
\[
v = \lambda_1^{-a_1}\lambda_2^{-a_2}\cdots\lambda_n^{-a_n} w \lambda_n^{b_n} \cdots \lambda_1^{b_1},
\]
for some $w \in L_n$ and some non-negative integers $a_i$ and $b_i$ for $1 \le i \le n$.
 
 For any fixed $i$, the homomorphism
 \begin{align*}
  L_n \ast_{\tilde\lambda_1, \theta_1} \ast_{\tilde\lambda_2, \theta_2} \cdots \ast_{\tilde\lambda_n, \theta_n} &\longrightarrow \ZZ \\
  L_n &\longmapsto 0 \\
  \tilde\lambda_j &\longmapsto \left\{\begin{array}{l l}
                           1 & \text{if $i = j$,} \\
                           0 & \text{else,}
                          \end{array}\right.
 \end{align*}
 factors through $\Stab_V((0^{n-2}, \ldots, \varepsilon, \zeta))$ using Lemma \ref{lemma:properties of tildechi}, thus $a_i = b_i$ for all $i$.  In particular, $w \in \Ker \Phi$, thus $w = 1$, since $\Phi$ is the identity when restricted to $L_n$.
\end{proof}

\begin{Cor}\label{cor:V acts on Deltan with Finfty stab}
 Given $(x_1, \ldots, x_n) \in \Delta_n$, the group $\Stab_V((x_1, \ldots, x_n))$ is of type $\F_\infty$.
\end{Cor}
\begin{proof}
Since $V$ acts transitively on $\Delta_n$ it is sufficient to check that the stabiliser $\Stab_V((0^{n-2}, \ldots, \varepsilon, \zeta))$ is $\F_\infty$.  This follows since by Lemma \ref{Lemma:Ln iso V} $L_n$ is isomorphic to $V$ and hence $\F_\infty$ and HNN-extensions of groups of type $\F_\infty$ are themselves $\F_\infty$ \cite[Proposition 2.13(b)]{Bieri-HomDimOfDiscreteGroups}.
\end{proof}

\section{Normal subgroups}\label{section:normal subgroups}

In this section we give a complete description of the normal subgroups of $QF$, $QT$, $QV$, $\tilde{Q}T$, and $\tilde{Q}V$.  

Recall that $\Sym(\{0,1\}^*)$ has only one non-trivial proper normal subgroup, this is the finite support alternating group $\Alt(\{0,1\}^*)$, the subgroup of all permutations with even parity.  Moreover, $\Sym(\{0,1\}^*)$ and $\Alt(\{0,1\}^*)$ are normal in the infinite support symmetric group on $\{0,1\}^*$, because conjugating any finite permutation by any bijection $\{0,1\}^* \to \{0,1\}^*$ preserves the cycle type of the finite permutation.  

Recall also that $F/[F,F] \cong \ZZ \oplus \ZZ$ but, since $T$ and $V$ are simple and non-abelian, $[T,T] \cong T$ and $[V,V] \cong V$ \cite[Theorems 4.1,5.8,6.9]{CannonFloydParry-IntroNotesOnThompsons}.

\begin{Theorem}[Normal subgroups of $QF$]\label{theorem:alternative for normal subgroups}
~
\begin{enumerate}
 \item A non-trivial normal subgroup of $QF$ is either $\Alt(\{0,1\}^*)$, $\Sym(\{0,1\}^*)$, or contains 
 \[
 [QF,QF] = \Alt(\{0,1\}^*) \rtimes [F,F].
 \]
 \item A proper non-trivial normal subgroup of $\tilde{Q}T$ is either $\Alt(Z)$, $\Sym(Z)$, or 
 \[
 [\tilde{Q}T, \tilde{Q}T] = \Alt(Z) \rtimes T.
 \]
 \item A proper non-trivial normal subgroup of $\tilde{Q}V$ is either $\Alt(Z)$, $\Sym(Z)$, or 
 \[
 [\tilde{Q}V, \tilde{Q}V] = \Alt(Z) \rtimes V.
 \]
 \item A proper non-trivial normal subgroup of $QT$ is one of either $\Alt(\{0,1\}^*)$, $\Sym(\{0,1\}^*)$, or 
 \[
 [QT,QT] = (\Alt(Z) \rtimes T) \cap QT.
 \]  
 Moreover, $(\Alt(Z) \rtimes T) \cap QT$ is an extension of $T$ by $\Sym(\{0,1\}^*)$.
 \item A proper non-trivial normal subgroup of $QV$ is one of either $\Alt(\{0,1\}^*)$, $\Sym(\{0,1\}^*)$, or 
 \[
 [QV, QV] = (\Alt(Z) \rtimes V) \cap QV.
 \]  
 Moreover, $(\Alt(Z) \rtimes V) \cap QV$ is an extension of $V$ by $\Sym(\{0,1\}^*)$.
\end{enumerate}
\end{Theorem}

This theorem is similar to results for the braided Thompson groups $F_{\text{br}}$ and $V_{\text{br}}$.  The braided Thompson group $F_{\text{br}}$ is a split extension with kernel the group $P_{\text{br}}$, a direct limit of finite pure braid groups, and quotient $F$.  The braided Thompson group $V_{\text{br}}$ is a non-split extension with kernel $P_{\text{br}}$ and quotient $V$.  Zaremsky shows that for any normal subgroup $N$ of $F_{\text{br}}$, either $N \le P_{\text{br}}$ or $[F_{\text{br}}, F_{\text{br}}] \le N$ \cite[Theorem 2.1]{Zaremsky-NormalSubgroupsOfBraidedThompsonsGroups}, and a proper normal subgroup $N$ of $V_{\text{br}}$ is necessarily contained in $P_{ \text{br}}$ \cite[Corollary 2.8]{Zaremsky-NormalSubgroupsOfBraidedThompsonsGroups}.  There is no braided version of Thompson's group $T$.  We borrow some of the methods in this section from \cite[\S 2]{Zaremsky-NormalSubgroupsOfBraidedThompsonsGroups}.

The next proposition should be compared with \cite[Theorem 5]{Ore-SomeRemarksOnCommutators}, that every element of the full permutation group on an infinite set is a commutator, so in particular the commutator subgroup of the full permutation group on $\{0,1\}^*$ is itself.

We write $\llangle H \rrangle_G$ for the normal closure in $G$ of a subgroup $H \le G$, and for an element $g \in G$ we write $\llangle g \rrangle_G$ for the normal subgroup generated by $g$.

\begin{Prop}\label{prop:normal closure of iF commutator}
~
\begin{enumerate}
 \item
 $
    \llangle [\iota F,\iota F] \rrangle_{QF} = [QF,QF] = \Alt(\{0,1\}^*) \rtimes [F,F].
 $
 \item 
 $
    \llangle [\iota F,\iota F] \rrangle_{\tilde{Q}T} = [\tilde{Q}T,\tilde{Q}T] = \Alt(Z) \rtimes T.
 $
 \item 
 $
    \llangle [\iota F,\iota F] \rrangle_{\tilde{Q}V} = [\tilde{Q}V,\tilde{Q}V] = \Alt(Z) \rtimes V. 
 $
 \item $\llangle [\iota F,\iota F] \rrangle_{QT} = [QT,QT] = (\Alt(Z) \rtimes T) \cap QT$, moreover this subgroup is an extension with kernel $\Sym(\{0,1\}^*)$ and quotient $T$.
 \item 
 $\llangle [\iota F,\iota F] \rrangle_{QV} = [QV,QV] = (\Alt(Z) \rtimes V) \cap QV$, moreover this subgroup is an extension with kernel $\Sym(\{0,1\}^*)$ and quotient $V$.
\end{enumerate}
\end{Prop}
\begin{proof}For part (1), consider $[\alpha, \beta] \in [\iota F, \iota F]$ then for $\sigma_{1,11} \in \Sym(\{0,1\}^*)$, we have 
\[
\sigma_{1,11} [\alpha,\beta]^{-1} \sigma_{1,11} [\alpha,\beta] \in  \llangle [\iota F,\iota F] \rrangle. 
\]
Calculating explicitly,
\[
\sigma_{1,11} [\alpha,\beta]^{-1} \sigma_{1,11} [\alpha,\beta] = \sigma_{1,11,1110}.
\]
The smallest normal subgroup of $\Sym(\{0,1\}^*)$ containing a $3$-cycle is $\Alt(\{0,1\}^*)$ so $\Ker \pi \vert_{\llangle [\iota F, \iota ,F] \rrangle} = \Alt(\{0,1\}^*) $.  Combining this with the fact that $[F,F] \le \pi (\llangle [\iota F, \iota F] \rrangle )$ gives  
\[
 \Alt(\{0,1\}^*) \rtimes [F,F] \le \llangle [\iota F,\iota F] \rrangle.
\]

Since $\llangle [\iota F,\iota F] \rrangle_{QF} \le [QF, QF]$, we know that $\Alt(\{0,1\}^*) \rtimes [F,F] \le [QF, QF]$ and we now show they are equal.

Consider the following map induced by $\pi$,
\[
 \pi^\prime: QF / (\Alt(\{0,1\}^*) \rtimes [F,F]) \longrightarrow F/[F,F].
\]
 Applying $\pi^\prime$ to the quotient $[QF, QF] / (\Alt(\{0,1\}^*) \rtimes [F,F])$ gives the trivial group.  Since the kernel of $\pi^\prime $ is contained in $\Sym(\{0,1\}^*) / \Alt(\{0,1\}^*) = C_2$, the quotient $[QF, QF] / (\Alt(\{0,1\}^*) \rtimes [F,F])$ is either trivial or $C_2$.  However, if the quotient $[QF, QF] / (\Alt(\{0,1\}^*) \rtimes [F,F])$ was $C_2$ then $[QF, QF] = \Sym(\{0,1\}^*) \rtimes [F,F]$ which would contradict Corollary \ref{cor:abelianisation QF}.  Hence $[QF, QF] = \Alt(\{0,1\}^*) \rtimes [F,F]$.  This completes the proof of (1).
 
For parts (2) and (3) substitute Corollary \ref{cor:abelianisation wQT} or Corollary \ref{cor:abelianisation wQV} for Corollary \ref{cor:abelianisation QF} and use that $[T,T] \cong T$ and $[V,V] \cong V$ \cite[Theorems 5.8,6.9]{CannonFloydParry-IntroNotesOnThompsons}.

For part (5), by the argument at the beginning of this proof, $\Alt(\{0,1\})^*$ is contained in $\llangle [\iota F, \iota F] \rrangle_{QV}$ and $\llangle [\iota F, \iota F] \rrangle_{QV}$ cannot contain any finite cycle with odd parity, else $\llangle [\iota F, \iota F] \rrangle_{\tilde{Q}V}$ would also, and we already know that it doesn't.  Also, $\llangle [\iota F , \iota F] \rrangle_{QV}$ projects onto a non-trivial normal subgroup of $V$, so it must project onto $V$.  We conclude that $
\llangle [\iota F , \iota F] \rrangle_{QV}$ is an extension of $V$ by $\Alt(\{0,1\}^*)$.  Observe that, 
\[
\llangle [\iota F , \iota F] \rrangle_{QV} \le \left( \llangle [\iota F , \iota F] \rrangle_{\tilde{Q}V} \right) \cap QV = \left( \Alt(Z) \rtimes V \right) \cap QV.
\]
Since $\left( \Alt(Z) \rtimes V \right) \cap QV$ is also an extension of $V$ by $\Alt(\{0,1\}^*)$, we necessarily have 
\[
 \llangle [\iota F , \iota F] \rrangle_{QV} = \left( \Alt(Z) \rtimes V \right) \cap QV.
\]

The proof of (4) is similar to that of (5).
\end{proof}

\begin{Cor}\label{cor:abelianisations from normal subgroup structure}
The abelianisation of $QF$ is isomorphic $\ZZ \oplus \ZZ \oplus C_2$ and the abeliansations of $QT$, $QV$, $\tilde{Q}T$, and $\tilde{Q}V$ are all isomorphic to $C_2$.
\end{Cor}

\begin{Lemma}\label{lemma:buildup}
~
\begin{enumerate}
 \item Let $\sigma \in \Sym(\{0,1\}^*)$ and $1 \neq f \in \iota F$ be arbitrary, then there exists $h \in \iota F$ such that $[h, \sigma] = 1$ but $[h,f] \neq 1$. 
 \item Let $\sigma \in \Sym(Z)$ and $1 \neq v \in \iota V$ be arbitrary, then there exists $h \in \iota F$ such that $[h, \sigma] = 1$ by $[h,f] \neq 1$.
\end{enumerate}
\end{Lemma}
\begin{proof}
Let $T$ be the full sub-tree under some $x \in \{0,1\}^*$, where $T$ is chosen to intersect the support of $f$ but not the support of $\sigma$, this is possible since the support of $f$ is infinite but the support of $\sigma$ is finite.  Thus any $h$ with $\supp(h) \subseteq T$ satisfies $[h,\sigma] = 1$.

Let $F^\prime$ denote the copy of $F$ acting on $T$, with the usual generators denoted $\alpha^\prime$ and $\beta^\prime$.  If $f$ restricts to an action on $T$ then, since $F^\prime$ is centreless, there exists some $h \in F^\prime$ which doesn't commute with $f$.  If $f$ doesn't restrict to an action on $T$ then there is necessarily an orbit $\{x_i\}_{i \in \ZZ}$ which is neither contained in $T$ nor its complement.  Since $\alpha^\prime$ acts non-trivially on all of $T$, one of the $\alpha^\prime$ orbits intersects $\{x_i\}_{i \in \ZZ}$ but is not equal to $\{x_i\}_{i \in \ZZ}$, thus $\alpha^\prime$ and $h$ do not commute.

The proof of (2) is analagous.
\end{proof}

\begin{Lemma}\label{lemma:normal closure of element in iF}
Let $G$ be one of $QF$, $QT$, $QV$, $\tilde{Q}T$, and $\tilde{Q}V$, and let $f \in G$ such that $\pi(f) \neq 1$, then $\llangle f \rrangle_G \cap \iota F \neq \{\id\}$.
\end{Lemma}
\begin{proof}
Let $\sigma f \in QF$, so that $\sigma \in \Sym(\{0,1\}^*)$ and $f \in \iota F$ with $\pi(f) \neq 1$.  Let $h \in \iota F$ be as in Lemma \ref{lemma:buildup}(1), then 
\[
 \sigma h f h^{-1} = h\sigma f h^{-1} \in \llangle \sigma f \rrangle.
\]
So,
\[
 [f,h] = (\sigma f)^{-1}\sigma h f h^{-1} \in \llangle \sigma f \rrangle,
\]
since $[f,h]$ is a non-trivial element of $\iota F$, this is sufficient.  This proves the statement for $QF$.

For $G = \tilde{Q}T$ or $\tilde{Q}V$, start with an element $\sigma t \in \tilde{Q}T$ (respectively $\sigma v \in \tilde{Q}V$), so that $\sigma \in \Sym(Z)$ and $t \in \iota T$ (resp. $t \in \iota V$).  Let $h \in \iota F$ be as in Lemma \ref{lemma:buildup}(2), then the proof is as for $QF$.

For $G = QT$ or $QV$, start with $\sigma t \in QT$ (respectively $\sigma v \in QV$), so that $\sigma \in \Sym(\{0,1\}^*)$ and $t \in \iota T$ (resp. $t \in \iota V$).  Let $h \in \iota F$ be as in Lemma \ref{lemma:buildup}(1), then the proof is as for $QF$.
\end{proof}

\begin{proof}[Proof of Theorem \ref{theorem:alternative for normal subgroups}]
Let $N$ be a normal subgroup of $QF$ with $N \not\le \Sym(\{0,1\}^*)$, then using Lemma \ref{lemma:normal closure of element in iF} $N$ contains an element $f \in \iota F$.  For any $g \in \iota F$, we have $g^{-1}f^{-1}g \in N$ and hence 
\[
 [g,f] = g^{-1}f^{-1}gf \in N.
\]
Without loss of generality, $f \in [\iota F, \iota F]$.  

Since $[\iota F, \iota F]$ is simple \cite[Theorem 4.5]{CannonFloydParry-IntroNotesOnThompsons}, we deduce that $[\iota F, \iota F] \le N$, and thus $\llangle [\iota F, \iota F] \rrangle \le N$.  Proposition \ref{prop:normal closure of iF commutator}(1) completes the proof.

For the other parts, start with an element $f$ of either $\tilde{Q}T$, $\tilde{Q}V$, $QT$, or $QV$ and then use the appropriate parts of Lemmas \ref{lemma:normal closure of element in iF} and Proposition \ref{prop:normal closure of iF commutator}. 
\end{proof}

\section{Bieri--Neumann--Strebel--Renz invariants}\label{section:BNSR invariants}
In this section we compute the Bieri--Neumann--Strebel--Renz invariants $\Sigma^i(QF)$ and $\Sigma^i(QF, R)$ for any commutative ring $R$.  The invariant $\Sigma^1(G)$ was introduced by Bieri, Neumann, and Strebel in \cite{BieriNeumannStrebel-AGeometricInvariantOfDiscreteGroups} and the higher invariants $\Sigma^i(QF)$ for $i \ge 2$ were introduced in \cite{BieriRenz-HigherGeometricInvariants}.  In general, for any group $G$, there is a hierachy of invariants
\[
 \Sigma^1(G) \supseteq \Sigma^2(G) \supseteq \cdots \supseteq \Sigma^i(G) \supseteq \cdots,
\]
furthermore we set $\Sigma^\infty(G) = \bigcap_{i = 1}^\infty \Sigma^i(G)$.  There are also homological versions $\Sigma^i(G, R)$ for any commutative ring $R$ fitting into a similar hierachy.  Furthermore, $\Sigma^i(G) \subseteq \Sigma^i(G, R)$ for all $i$ all commutative rings $R$ and also $\Sigma^1(G, R) = \Sigma^1(G)$ for all commutative rings $R$.

Some of the interest in the Bieri--Neumann--Strebel--Renz invariants comes from the following theorem, which classifies the finiteness length of normal subgroups $N$ of a group $G$ which contain the commutator subgroup $[G,G]$.
\begin{Theorem}[{\cite[Theorem B]{BieriRenz-HigherGeometricInvariants}\cite{Renz-GeometricInvariantsAndHNNExtensions}}]\label{theorem:class of normal subgroups via BNSR}
 Let $G$ be a group of type $\F_n$ (respectively $\FP_n$ over $R$) and let $N$ be a normal subgroup containing the commutator $[G,G]$.  Then $N$ is $\F_n$ (resp. $\FP_n$ over $R$) if and only if every non-zero character $\chi$ of $G$ such that $\chi(N) = 0$ satisfies $[\chi] \in \Sigma^n(G)$ (resp. $[\chi] \in \Sigma^n(G, R)$).
\end{Theorem}

Recall that a \emph{character} of a group $G$ is a group homomorphism $G \to \RR$, where $\RR$ is viewed as a group under addition, and the \emph{character sphere} $S(G)$ is the set of equivalence classes of non-zero characters modulo multiplication by a positive real number.  Thus $S(G)$ is isomorphic to a sphere of dimension $n-1$ where $n$ is the torsion-free rank of $G / [G,G]$.

Since $F / [F,F] \cong \ZZ \oplus \ZZ$, the character sphere $S(F)$ is isomorphic to $S^1$.  We denote by $\chi_0$ and $\chi_1$ the two characters of $F$ given by $\chi_0(A) = -1$, $\chi_0(B) = 0$, $\chi_1(A) = 1$, and $\chi_1(B) = 1$, these two characters are linearly independent and hence do not represent antipodal points on $S(F)$.  

\begin{Theorem}[The Bieri--Neumann--Strebel--Renz invariants of $F$ {\cite{BieriGeogheganKochloukova-SigmaInvariantsOfF}}]\label{theorem:BNSR invariants of F}
The character sphere $S(F)$ is isomorphic to $S^1$ and, for any commutative ring $R$,
\begin{enumerate}
 \item $\Sigma^1(F, R) = \Sigma^1(F) = S(F) \setminus \{ [\chi_0], [\chi_1]\}$.
 \item $\Sigma^i(F, R) = \Sigma^i(F) = S(F) \setminus \{ [a \chi_0 + b \chi_1] : a, b \ge 0 \}$ for all $i \ge 2$.
\end{enumerate}
\end{Theorem}

Since the abelianisation of $QF$ also has torsion-free rank $2$, the character sphere $S(QF)$ is again isomorphic to $S^1$.  By pre-composing $\chi_0$ and $\chi_1$ with $\pi : QF \to F$ we obtain characters $\pi^*\chi_0$ and $\pi^*\chi_1$.  Once again these are linearly independent.  In this section we prove the following theorem.

\begin{Theorem}[The Bieri--Neumann--Strebel--Renz invariants of $QF$]\label{theorem:BNSR invariants of QF}
The character sphere $S(QF)$ is isomorphic to $S^1$ and, for any ring $R$,
\begin{enumerate}
 \item $\Sigma^1(QF, R) = \Sigma^1(QF) = S(QF) \setminus \{ [\pi^*\chi_0], [\pi^*\chi_1]\}$.
 \item $\Sigma^i(QF, R) = \Sigma^i(QF) = S(QF) \setminus \{ [a \pi^*\chi_0 + b \pi^*\chi_1] : a, b \ge 0 \}$ for all $i \ge 2$.
\end{enumerate}
\end{Theorem}

\begin{Remark}[$\nu$-symmetry]\label{remark:nu-symmetry}
For $x \in \{0,1\}^*$, let $\bar{x}$ be the element obtained by swapping $0$ and $1$ in the word $x$.  The map $x \mapsto \bar{x}$ induces an automorphism $\nu : QV \to QV$, which projects to the automorphism of $F$ denoted $\nu$ in \cite[\S 1.4]{BieriGeogheganKochloukova-SigmaInvariantsOfF}.  Furthermore, $\nu$ induces an automorphism $\nu^*$ of $\Sigma^i(QF)$ for all $i$, and one checks that $\nu^* [\chi_0] = [\chi_1]$.  Following Bieri, Geoghegan, and Kochloukova we will call this the \emph{$\nu$-symmetry} of $\Sigma^i(QF)$.  Recall that for any group $G$ the invariants $\Sigma^i(G)$ and $\Sigma^i(G, R)$ are invariant under automorphisms of $G$.
\end{Remark}

\begin{Lemma}[{\cite[Corollary 2.8]{Meinert-ActionsOn2ComplexesAndSigma2}\cite[Corollary 3.12]{Meinert-HomologicalInvariantsForMetabelianGroups}}]\label{lemma:sigma inv under quotients}
 Let $\pi:G \longtwoheadrightarrow Q$ be an split surjection and $\chi$ a character of $Q$.  Then if $[\pi^*\chi] \in \Sigma^i(G)$ (respectively $\Sigma^i(G, R)$) then $[\chi] \in \Sigma^i(Q)$ (resp. $\Sigma^i(Q, R)$). 
\end{Lemma}

\begin{Theorem}[{\cite[Theorems 2.1(1), 2.3]{BieriGeogheganKochloukova-SigmaInvariantsOfF}}]\label{theorem:BGK thms on Sigma}
Let $H$ be a group of type $F_\infty$ and let $G = H\ast_{\theta, t}$ be an ascending HNN-extension such that $\theta $ is not surjective.
\begin{enumerate}
 \item Let $\chi : G \to \RR$ be the character given by $\chi(H) = 1$ and $\chi(t) = 1$, then $[\chi] \in \Sigma^\infty(G)$.
 \item Let $\chi : G \to \RR$ be a character such that $\chi\vert_H \neq 0$ and $[\chi\vert_H] \in \Sigma^\infty(H)$ then $[\chi] \in \Sigma^\infty(G)$, for any $i \ge 0$.
\end{enumerate}
\end{Theorem}

Recall that the characters $\pi^*\chi_i$ are completely determined by the values they take on $\alpha$ and $\beta$ (defined in Remark \ref{remark:explicit description of alpha beta etc}) and satisfy $\pi^*\chi_i(\alpha) = \chi_i(A)$ and $\pi^*\chi_i(\beta) = \chi_i(B) $ for all $i$, where $A$ and $B$ are the standard generators of $F$.

\medskip\noindent We now show that $QF$ can be written as an ascending HNN-extension similarly to those of  $F$  \cite[Proposition 1.7]{BrownGeoghegan-AnInfiniteDimensionTFFPinftyGroup} and  $BF$ \cite[Lemma 1.4]{Zaremsky-NormalSubgroupsOfBraidedThompsonsGroups}.
 
Let $QF(1)$ be the subgroup of $QF$ which fixes all of $\{0,1\}^*$ except the subtree with root $1$.  It is easy to see that $QF(1)$ is isomorphic to $QF$. Using Theorem \ref{theorem:QF presentation} we see that $QF$ is generated by $\{\alpha,\beta, \sigma_{\varepsilon, 1}\}.$ This shows that  $QF(1)$ is generated by $\beta$, $\beta^\alpha$ and $\sigma_{1,11}=\sigma_{\varepsilon, 1}^{\alpha}$.  Moreover, conjugating $QF(1)$ by $\alpha$ maps $QF(1)$ isomorphically onto the subgroup of $QF$ which fixes all but the subtree with root $11$.

\begin{Lemma}\label{lemma:QF is an HNN extension}
 The group $QF$ can be written as an ascending HNN-extension $QF = QF(1) \ast_{\theta, t}$, where $QF(1) \cong QF$.
\end{Lemma}
The proof is similar to \cite[Lemma 1.4]{Zaremsky-NormalSubgroupsOfBraidedThompsonsGroups}.
\begin{proof}
 Let $\theta : QF(1) \to QF(1)$ be the monomorphism $\tau \mapsto \alpha^{-1} \tau \alpha$ and let $\psi : QF(1) \ast_{\theta, t} \to QF$ be the map given by setting $t$ to be $\alpha$.  The map $\psi$ is surjective since $\psi(\beta) = \beta$, $\psi( t ) = \alpha$,  and $\psi(\sigma_{1,11}^{t^{-1}})=\sigma_{\varepsilon, 1}.$ 
 
 Let $g \in \Ker \psi$, since $QF(1) \ast_{\theta, t}$ is an ascending HNN-extension we can write $g$ in normal form as $g = t^nht^{-m}$ where $n, m \ge 0$ and $h \in QF(1)$.  Let $\chi : QF(1) \ast_{\theta, t} \longrightarrow \ZZ$ be the homomorphism sending $t$ to $1$ and sending $QF(1)$ to $0$.  The homomorphism $\chi$ factors through $\psi$---one can check there is a homomorphism $QF \to \ZZ$ sending $\alpha$ to $1$ and $\beta, \sigma_{\varepsilon,1}$ to $0$.  Thus, if $g \in \Ker \psi$ then $g \in \Ker \chi $ and so $n = m$.  In particular, $\alpha^n \psi(h) \alpha^{-n} = \id_{QF}$ so $h \in \Ker \psi$.  However $\psi$ is the identity on $QF(1)$, thus $h =1$ and $\psi$ in injective.
\end{proof}

\begin{proof}[Proof of Theorem \ref{theorem:BNSR invariants of QF}]
  Lemma \ref{lemma:sigma inv under quotients} implies that for all $i$, we have $\Sigma^i(QF, R) \subseteq \Sigma^i(F, R)$.  In particular, 
  \[
   \Sigma^2(QF, R) \subseteq S(QF) \setminus \{ [a \pi^*\chi_0 + b \pi^*\chi_1] : a, b \ge 0 \}
  \]

  The character which takes $1$ on $\alpha$ and $0$ elsewhere is exactly $-\pi^*\chi_0$ so using Theorem \ref{theorem:BGK thms on Sigma}(1) we find that $[-\pi^* \chi_0] \in \Sigma^\infty(QF)$ and by Remark \ref{remark:nu-symmetry}, $[-\pi^* \chi_1] \in \Sigma^\infty(QF)$ as well. 
  
  Next we claim that
 \[
  \{[\chi] \in S(QF) : \chi(\beta) < 0 \} \subseteq \Sigma^\infty(QF),
 \]
 the argument is similar to that of \cite[Corollary 2.4]{BieriGeogheganKochloukova-SigmaInvariantsOfF}.  Let $\chi: QF \to \RR$ be a character with $\chi(\beta) < 0$, then
 \[
  \psi^* [\chi\vert_{QF(1)}] (\alpha) = \chi(\beta) < 0,
 \]
\[
  \psi^* [\chi\vert_{QF(1)}] (\beta) = \chi(\beta^\alpha) = \chi(\beta) < 0.
 \]
 So $\psi^* [\chi\vert_{QF(1)}] = -[\chi_1]$ and thus $[\chi\vert_{QF(1)}] \in \Sigma^\infty(QF(1))$.  Since $QF(1)$ is isomorphic to $QF$ and hence of type $\F_\infty$ (Theorem \ref{theorem:QF is Finfty}), applying Theorem \ref{theorem:BGK thms on Sigma}(2) completes the claim. 
 
 The above implies that $[a\chi_0 + b \chi_1] \in \Sigma^\infty(QF)$ if $b < 0$, so by $\nu$ symmetry $[a\chi_0 + b \chi_1] \in \Sigma^\infty (QF)$ if $a < 0$ also.  Thus,
 \[
   \{ [a \pi^*\chi_0 + b \pi^*\chi_1] : a, b \ge 0 \} \subseteq \Sigma^\infty(QF)^c
 \]
 which completes the descriptions of $\Sigma^i(QF)$ and $\Sigma^i(QF, R)$ for all $i \ge 2$.

 It remains to complete the description of $\Sigma^1(QF)$, we use Theorem \ref{theorem:class of normal subgroups via BNSR} for this.  Let $\chi \in S(QF) \setminus \{[\pi^*\chi_0], [\pi^*\chi_1]\}$ and consider the normal subgroup $N = \Ker \pi^*\chi$.  We will show that $N$ is finitely generated, so necessarily $\chi \in \Sigma^1$.
 
 From Theorem \ref{theorem:BNSR invariants of F} we know that $\pi(N)$ is finitely generated.  Since $\Ker \pi\vert_N = \Sym(\{0,1\}^*)$ we have a group extension
 \[
  1 \longrightarrow \Sym(\{0,1\}^*) \longrightarrow N \longrightarrow \pi(N) \longrightarrow 1.
 \]
 Thus, using the method of Section \ref{section:finite presentations} it is sufficient to find a generating set of $\Sym(\{0,1\}^*)$ on which $\pi(N)$ acts by permutations and with finitely many orbits, this is given by Lemma \ref{lemma:KerChi acts with finitely many orbits on Sigma2} below, since a generating set for $\Sym(\{0,1\}^*)$ is given by the set $\Sigma_2$, which we recall below.
\end{proof}

Recall from Section \ref{subsection:finite presentation for QF} that for all $n$,
\[
  \Sigma_n = \left\{ (x_1, \ldots, x_n) \in \prod_1^n \{0,1\}^* : x_1 \lneq_{\lex} x_2 \lneq_{\lex} \cdots \lneq_{\lex} x_n\right\},
\]
and $\Sigma_n$ admits an action of $F$ via the splitting $\iota:F \longrightarrow QF$, the inclusion $QF \longrightarrow QV$, and the usual action of $QV$ on $\{0,1\}^*$.

Recall the definition of $\tilde\chi_\zeta$ from Section \ref{subsection:Stab is Finfty}.  By comparing the values they take on $\alpha$ and $\beta$, one checks that $\pi^* \chi_1 = \tilde\chi_\zeta$.

\begin{Lemma}\label{lemma:KerChi acts with finitely many orbits on Sigma2}
 Let $\chi \in S(QF)$ and $N = \Ker \chi$ then $\pi(N)$ acts transitively on $\Sigma_2$.
\end{Lemma}

\begin{proof}
 Let $(x_1, x_2) \in \Sigma_2$ and let $(L, R)$ be a representative for an element $f$ of $F$ such that $\iota(f)(x_1) = 0$ and $\iota(f)(x_2) = \varepsilon$ (use, for example, Lemma \ref{lemma:transitive action Sigman}).  Applying Lemma \ref{lemma:arbitrary chi0 and chi1 while fixing points} below with $a = -\chi_1 (f)$ gives an element $f_1$, and applying the lemma a second time with $a = -\chi_0 (f)$ gives a second element $f_0$.
 
 We claim that $\nu(f_0) f_1 f$ is the required element of $F$.  Note that since $f_0$ fixes $\varepsilon$ and $1$ by construction, the element $\nu(f_0)$ fixes $\bar 1 = 0$ and $\bar \varepsilon = \varepsilon$.  Finally,
 \[
  \chi_0 (\nu(f_0) f_1 f) = \chi_0(\nu(f_0)) + \chi_0(f_1) + \chi_0(f) = \chi_1(f_0) + \chi_0(f) = 0,
 \]
 and similarly one shows that $\chi_1 (\nu(f_0) f_1 f) = 0$.
 
 Let $\tau = \iota(\nu(f_0) f_1 f)$.  Since $\chi_1(\nu(f_0) f_1 f) = \chi_0(\nu(f_0) f_1 f) = 0$, also
 $  \pi^*\chi_1(\tau) = \pi^*\chi_0(\tau) = 0$, so the element $\tau$ is in the kernel of any character in $S(QF)$.
\end{proof}

\begin{Lemma}\label{lemma:arbitrary chi0 and chi1 while fixing points}
For any integer $a$ there exists an element $f \in F$ such that $\iota(f) \cdot \varepsilon = \varepsilon$, $\iota(f) \cdot 0  = 0 $, $\iota(f) \cdot 1 = 1$, $\chi_0\iota(f) = 0$, and $\chi_1 \iota(f) = a$.
\end{Lemma}
\begin{proof}
 Assume for now that $a > 0$.  Let $T$ be the smallest tree containing $\varepsilon$, $1$, and $11$ as nodes.  Let $L$ be obtained from $T$ by attaching $a$ carets iteratively to the largest leaf of $T$ (using the $\le_{\lex}$ ordering).  Let $R$ be obtained from $T$ by attaching $a$ carets iteratively to the second largest caret of $T$.  Denote by $f$ the element of $F$ with tree diagram representative $(L, R)$.  For example when $a = 2$ we obtain the element of Figure \ref{figure:f when a is 2}.  One checks that $\chi_1(f) = a$ and that $\iota(f)$ fixes $\varepsilon$, $0$, and $1$.  
 
 If $a<0$ then perform the steps above for $-a$ and then replace $f_0$ by $f_0^{-1}$.
\end{proof}

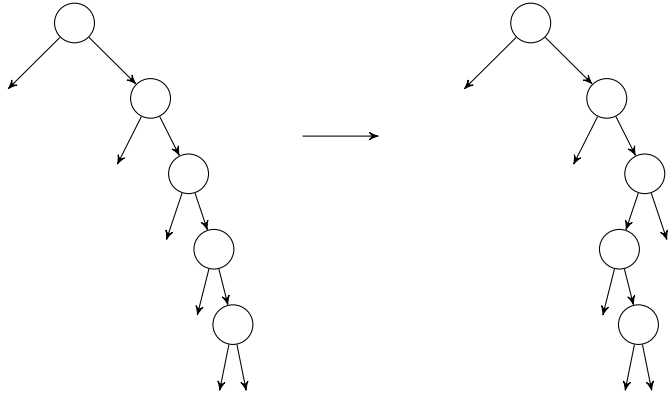
\begin{figure}[ht]
\begin{tikzpicture}[->,>=stealth',level/.style={sibling distance = 2cm/#1,
  level distance = 1cm}] 
  \node [arn_n] at (0,0) {}
    child{ node {} }
    child{ node [arn_n] {}
       child{ node {} }
       child{ node [arn_n] {} 
          child {node {} }
          child {node [arn_n] {} 
             child {node {} }
             child {node [arn_n] {} 
               child { node {} }
               child { node {} }
             }
          }
       }
    }; 
  \draw (3,-1.5) -- (4,-1.5);
 
  \node [arn_n] at (6,0) {}
    child{ node {} }
    child{ node [arn_n] {}
       child{ node {} }
       child{ node [arn_n] {} 
          child {node [arn_n] {} 
            child { node {} }
            child { node [arn_n] {}
               child { node {} } 
               child { node {} }
            }
          }
          child {node {} }
       }
    }; 
\end{tikzpicture}
\caption{Element $f$ from Lemma \ref{lemma:arbitrary chi0 and chi1 while fixing points} when $a = 2$.}
\label{figure:f when a is 2}
\end{figure}

\section{Relationships with other generalisations of Thompson's groups}\label{section:relationship of QF to others}

In \cite{Thumann-FinitenessForOperadGroups} Thumann studies the class of \emph{Operad groups}, this class contains the generalised Thompson groups $F_{n,r}$ and $V_{n,r}$, the higher dimensional Thompson groups $nV$, and the braided Thompson groups $BV$.  Moreover, the class of \emph{diagram groups} appearing in \cite{GubaSapir-DiagramGroups} may be described as Operad groups, as may the class of \emph{self-similarity groups} described by Hughes \cite{Hughes-LocalSimilaritiesAndTheHaagerupProperty}---for a proof of these facts see \cite[\S 3.5]{Thumann-FinitenessForOperadGroups}.  For each of these classes of groups there exist results showing that certain groups in the class are of type $\F_\infty$ \cite[Theorem 1.3]{Farley-FinitenessAndCAT0PropertiesOfDiagramGroups}\cite[Theorem 1.1]{FarleyHughes-FinitenessPropertiesOfSomeGroupsOfLocalSimilarities}\cite[Theorem 4.3]{Thumann-FinitenessForOperadGroups}.  We show in Section \ref{subsection:Thumann} that while $QV$ admits an easy description in all three of these classes, it cannot be proved to be of type $\F_\infty$ using these descriptions and the results listed above.

There is a generalisation of Thompson's groups due to Mart\'{\i}nez-P\'erez and Nucinkis, describing a family of automorphism groups of Cantor algebras \cite{MartinezNucinkis-GeneralizedThompsonGroups}.  We show in Proposition \ref{prop:QF is not a MartinezPerezNucinkis Group} that none of the groups studied in this paper appear in this family of groups.

Witzel and Zaremsky have introduced \emph{Thompson groups for systems of groups} \cite{WitzelZaremsky-ThompsonSystems}, this class contains the Thompson groups $F$ and $V$ and also the braided Thompson groups $BF$ and $BV$.  They study the finiteness lengths of some groups in this class.  We do not know of a way to express any of the groups studied here as a group in this class and we are unsure of the relationship between Witzel--Zaremsky's class of Thompson groups for systems of groups and either Thumann's class of Operad groups or the class of groups of Mart\'{\i}nez-P\'erez and Nucinkis.
 
\subsection{Operad groups}\label{subsection:Thumann}

Following \cite{Thumann-FinitenessForOperadGroups} we denote by 
\[
\operatorname{End}: \mathbf{MON} \rightarrow \mathbf{OP}
 \]
the endomorphism functor taking a strict monoidal category $\mathcal{C}$ to the endomorphism operad $\End(\mathcal{C})$.  Let $\mathcal{S}$ denote the left adjoint of $\End$.  Given a category $\mathcal{C}$ and an object $X \in \mathcal{C}$ we write $\pi_1(\mathcal{C}, X)$ for the fundamental group of $\mathcal{C}$ based at $X$.  The \emph{operad group} associated to an operad $\mathcal{O}$ and object $X \in \mathcal{S}(\mathcal{O})$ is the group $\pi_1(\mathcal{S}(\mathcal{O}), X)$.

Let $\mathcal{O}_{QV}$ be the operad with colours $l$ and $n$ (one should think of these as representing a leaf and a node of a finite subtree of $\mathcal{T}_{2,c}$) and with a single operation $\mathcal{O}(l,n,l;l) = \{\varphi\}$ (one should think of this operation as adding a caret).  
\begin{Lemma}
There is an isomorphism
 \[
  QV \cong \pi_1(\mathcal{S}(\mathcal{O}_{QV}), l).
 \]
\end{Lemma}
\begin{proof}
This is essentially \cite[Example 4.4]{FarleyHughes-BraidedDiagramGroupsAndLocalSimilarityGroups}, where Farley and Hughes describe $QV$ as a braided diagram group over the semi-group presentation  
\[
  \mathcal{P} = \langle l, n : (l, lnl) \rangle.
 \]    
 One can then convert this to a description of $QV$ as a group acting on a compact ultrametric space via a small similarity structure \cite[Theorem 4.12]{FarleyHughes-BraidedDiagramGroupsAndLocalSimilarityGroups} and then in turn to a description as an operad group \cite[\S 3.5]{Thumann-FinitenessForOperadGroups}.  Following this method one obtains the operad $\mathcal{O}_{QV}$ given above.
\end{proof}

 The next lemma demonstrates that one cannot use the result of Thumann  \cite[Theorem 4.3]{Thumann-FinitenessForOperadGroups}) to show that $QV$ is $\F_\infty$ because $\mathcal{O}_{QV}$ is not colour-tame, see \cite[Definition 4.2]{Thumann-FinitenessForOperadGroups}.   Recall that a colour word is said to be \emph{reduced} if no subword is in the domain of a non-identity operation.
 \begin{Lemma}
 $\mathcal{O}_{QV}$ is not colour-tame.
 \end{Lemma}
 \begin{proof}
  The operad $\mathcal{O}_{QV}$ is planar and has a finite set of colours, however the colour word 
  \[
   \overbrace{n \cdot n \cdots n}^{\text{$i$ times}},
  \]
  is reduced for all $i \ge 0$.
 \end{proof}

Using \cite[Theorem 4.12]{FarleyHughes-BraidedDiagramGroupsAndLocalSimilarityGroups} to convert the description of $QV$ as a braided diagram group over the semi-group presentation  
$\mathcal{P}$ to a description of $QV$ as a self-similarity group gives the following.  Let $\mathcal{T}_{\mathcal{P}}$ be the tree obtained from $\mathcal{T}_{2,c}$ by forgetting the colouring and adding to every node a single child.  Let $X$ be the ultrametric space $X = \Ends(\mathcal{T}_{\mathcal{P}})$ obtained by setting $d_X(p, p^\prime ) = e^{-i}$ if $p$ and $p^\prime$ are any two paths (without backtracking) which contain exactly $i$ edges in common.  Balls in $X$ are all of the form 
\[
 B_v = \{ p \in \Ends(\mathcal{T}_{\mathcal{P}}) : \text{$v$ lies on $p$}\},
\]
for some vertex in $v$.  If $v$ and $w$ are both leaves or both nodes then $\Sim_X (B_v, B_w)$ contains a single element and otherwise $\Sim_X(B_v, B_w)$ is empty.

The next lemma demonstrates that one cannot use the result of Farley and Hughes \cite[Theorem 1.1]{FarleyHughes-FinitenessPropertiesOfSomeGroupsOfLocalSimilarities} together with the description of $QV$ as the self-similarity group associated to the ultrametric space $X$ and the similarity structure $\Sim_X$ to prove that $QV$ is $\F_\infty$.  This is because  \cite[Theorem 1.1]{FarleyHughes-FinitenessPropertiesOfSomeGroupsOfLocalSimilarities} requires that $\Sim_X$ is rich in simple contractions \cite[Definition 5.11]{FarleyHughes-FinitenessPropertiesOfSomeGroupsOfLocalSimilarities}.

\begin{Lemma}
 $\Sim_X$ is not rich in simple contractions.
 \end{Lemma}
 \begin{proof}
  Given an arbitrary constant $k \in \NN_{>0}$ it suffices to exhibit a pseudo-vertex $v$ of height $k$ such that for every pseudo vertex $w \subseteq v$, either $\lVert w \rVert = 1$ or there is no simple contraction of $v$ at $w$.
  
  Let $k \in \NN_{>0}$ be an arbitrary constant and choose $k$ leaves $n_1, \ldots, n_k \in \Ends(\mathcal{T}_{\mathcal{P}})$ with no children.  Each of these leaves is a ball $B_{n_i}$ and so
  \[
  v = \{ [\operatorname{incl}_{B_{n_i}}, B_{n_i}] : i = 1, \ldots, k\}
  \]
  is a pseudo-vertex of height $k$.  Since any simple expansion of any pseudo-vertex necessarily introduces balls containing leaves with children, there can be no simple contraction of $v$.
 \end{proof}

 \subsection{Automorphisms of Cantor algebras}
  
  In \cite{MartinezNucinkis-GeneralizedThompsonGroups}, Mart\'{\i}nez-P\'erez and Nucinkis study certain automorphism groups $G_r(\Sigma)$, $T_r(\Sigma)$, and $F_r(\Sigma)$ of certain Cantor algebras. 
 
 \begin{Prop}\label{prop:QF is not a MartinezPerezNucinkis Group}
   None of $QV$, $\tilde{Q}V$, $QT$, $\tilde{Q}T$, and $QF$ are isomorphic to either $G_r(\Sigma)$, $T_r(\Sigma)$, or $F_r(\Sigma)$ for any $\Sigma$ and $r$.
 \end{Prop}
 \begin{proof}
  By \cite[Theorems 4.3, 4.8]{MartinezNucinkis-GeneralizedThompsonGroups}, any group $G_r(\Sigma)$ or $T_r(\Sigma)$ has at most finitely many conjugacy classes of finite subgroups isomorphic to a given finite subgroup. This is false however for $QV$ and $QT$ as one can find infintely many non-conjugate subgroups isomorphic to the cyclic group of order 2:  let $\{x_1, \ldots, x_n, \ldots\}$ and $\{y_1, \ldots, y_n, \ldots \}$ be two disjoint countably infinite subsets of $\{0,1\}^*$ and let $G_i$ be the subgroup of $QF$ which tranposes $x_j$ and $y_j$ for all $j \le i$.  Clearly $G_i \cong C_2$ for all $i$, but the $G_i$ are all pairwise non-conjugate (this is because conjugation preserves cycle type in infinite support permutation groups).  Thus none of $QV$, $\tilde{Q}V$, $QT$, $\tilde{Q}T$, and $QF$ may be isomorphic to $G_r(\Sigma)$ or $T_r(\Sigma)$.
  
  Since $F_r(\Sigma)$ is always torsion-free \cite[Remark 2.17]{MartinezNucinkis-GeneralizedThompsonGroups}, none of $QV$, $\tilde{Q}V$, $QT$, $\tilde{Q}T$, or $QF$ may be isomorphic to $F_r(\Sigma)$.
 \end{proof}  

\bibliographystyle{amsalpha}
\bibliography{main.bib}
\end{document}